\documentclass[math,dissertation, nochapterblankpages, nononchapterblankpages, plain]{puthesis}
\usepackage[pass,letterpaper]{geometry}

\usepackage{amsmath}
\usepackage{amssymb}

\usepackage{multicol}

\usepackage{subfigure}

\title{%
  Density of Self-Dual\\
  Automorphic Representations of $GL_n(\mathbb{A_Q})$%
}

\author{V\'\i t\v ezslav Kala}{Kala, V\'\i t\v ezslav}

\pudegree{Doctor of Philosophy}{Ph.D.}{August}{2014}

\majorprof{Freydoon Shahidi}

\campus{West Lafayette}

\newcommand{\depth}{\mathrm{depth\ }}
\newcommand{\Tr}{\mathrm{Tr}}

\newcommand{\aq}{\mathbb A}

\newcommand{\lieg}{\mathfrak g}
\newcommand{\lieh}{\mathfrak h}
\newcommand{\liea}{\mathfrak a}
\newcommand{\liegc}{\mathfrak g_{\mathbb C}}
\newcommand{\liehc}{\mathfrak h_{\mathbb C}}
\newcommand{\lieac}{\mathfrak a_{\mathbb C}}
\newcommand{\liesl}{\mathfrak {sl}}
\newcommand{\liesp}{\mathfrak {sp}}
\newcommand{\lieso}{\mathfrak {so}}

\usepackage{amssymb}
\sbox{\proofbox}{$\Box$}

\begin{document}


\volume

%
%
%
%
%


\begin{acknowledgments}
I want to greatly thank my advisor Professor Shahidi for all his valuable advice, support and encouragements during my work with him.
I wish to thank Professors Goldberg, Liu, and McReynolds for agreeing to serve on my committee and for their help and useful suggestions. 
I am grateful to Professors Arthur, Jiang, Rohrlich, and J.-K. Yu, and to Britain Cox, Radhika Ganapathy, Yeansu Kim, Jing Feng Lau, Manish Mishra, Abhishek Parab, and Sohei Yasuda, for interesting mathematical discussions and suggestions which have influenced and improved this thesis.
I also want to thank many other professors and staff at Purdue and Charles University for their
help and support. Let me mention only Professors Lipshitz and Dr\' apal, whose mentorship and support have been very important for me.
During my first three years at Purdue I have been fully supported by the International Fulbright Science and Technology Award -- I greatly appreciate this opportunity.
And finally, I thank my family and all my friends, both home in the Czech Republic and in the US. Without you it wouldn't have been so much fun!
\end{acknowledgments}


\tableofcontents

\listoftables






\begin{abstract}
We study the number $N_{\mathrm{sd}}^K(\lambda)$ of self-dual cuspidal automorphic representations of $GL_N(\mathbb{A_Q})$ which are $K$-spherical with respect to a fixed compact subgroup $K$ and whose Laplacian eigenvalue is $\leq \lambda$.
We prove Weak Weyl's Law for $N_{\mathrm{sd}}^K(\lambda)$ in the form that there are positive constants $c_1, c_2$ (depending on $K$) and $d$ such that
$c_1\lambda^{d/2}\leq N_{\mathrm{sd}}^K(\lambda)\leq c_2\lambda^{d/2}$ for all sufficiently large $\lambda$.
When $N=2n$ is even and $K$ is a maximal compact subgroup at all places, 
we prove Weyl's Law for the number of self-dual representations, i.e.,  $N_{\mathrm{sd}}^K(\lambda)=c\lambda^{d/2}+o(\lambda^{d/2})$.
These results are based on considering functorial descents of self-dual representations $\Pi$ to quasisplit classical groups $\mathbf G$.

In order to relate the properties of representations under functoriality, we discuss the infinitesimal character of the real component $\Pi_\infty$, which determines the Laplacian eigenvalue. To relate the existence of $K$-fixed vectors, we study the depth of $p$-adic representations, proving a weak version of depth preservation. We also consider the explicit construction of local descent, which allows us to improve the results towards depth preservation for generic representations.

\end{abstract}

%
%
%

\chapter{Introduction}

One of the important problems in modern number theory is the study of automorphic representations, i.e., representations which appear in the spectral decomposition of $L^2(\mathbf G(F)\backslash \mathbf G(\mathbb A_F))$ for a reductive group $\mathbf G$ over a number field $F$. 

Among the first incarnations of this idea was the theory of modular and Maass forms, which are essentially just automorphic forms on $GL_2(\mathbb{A_Q})$. 
It is not hard to construct examples of modular forms, but Maass forms are much harder to handle concretely -- in fact, at first it was not clear if they existed at all. Their existence and asymptotics were first proved by Selberg \cite{se} in the form of Weyl's Law, which he deduced from his trace formula.

In the classical setting, Weyl's Law provides the asymptotic estimates for the number of eigenvalues of the Laplacian operator on a compact Riemannian manifold. Automorphic forms are, among other things, eigenfunctions of the Laplacian on locally symmetric spaces $A_\mathbb R G(F)\backslash G(\mathbb A_F)/K$ for suitable compact subgroups $K$ of $G(\mathbb A_F)$. 
Weyl's Law is expected to hold in some generality even in this setting,
and has been proven in a number of cases, e.g., by M\" uller \cite{m} for $GL_N(\mathbb{A_Q})$ and by Lindenstrauss-Venkatesh \cite{lv} for split adjoint semisimple groups $\mathbf G$. See Section \ref{backgr weyl} for a precise statement and more background information on Weyl's Law.

The principle of functoriality, predicting transfers of automorphic representations between certain pairs of groups $G_1$ and $G_2$, plays a crucial role in the Langlands program. 
There has been quite a lot of progress recently in establishing functoriality from (quasisplit) classical groups to $GL_N$, first by Cogdell, Kim, Piatetski-Shapiro, and Shahidi \cite{ckpss} for generic representations and then by Arthur \cite{arthur} for all representations.
In these cases, the images of functorial lifts are among the self-dual automorphic representations of $GL_N(\mathbb{A_Q})$, as it follows from the work of Ginsburg, Rallis, and Soudry \cite{grs}. See Section \ref{backgr funct} for more information on functoriality. 

However, being self-dual is a fairly restrictive condition. Thus we may ask how many self-dual
representations are there? Do they have density zero among all representations?

In a local setting, Adler \cite{adler} gave explicit formulas for the number of depth zero self-dual supercuspidal representations of $GL_N$ over a local field.
For (global) Galois representations, 
Rohrlich \cite{rohrlich} has proved that the density is zero for representations 
$\mathrm{Gal}(\bar{\mathbb Q}/\mathbb Q)\rightarrow GL_N(\mathbb C)$ when $N=1$ or $2$ and obtained some conditional results for $N=3$.

The main result of our thesis is the following theorem, concerning the number of self-dual cuspidal automorphic representations of $GL_N(\mathbb{A_Q})$:

\begin{theorem}\label{intro main}
Let $N\in\mathbb N$.
For $M\in\mathbb N$ denote by $K(M)$ the principal congruence subgroup of $GL_N(\mathbb {A_Q})$, i.e., 
$K(M)=K_\infty\times\prod K_p$ with $K_\infty=O(N)$ and $K_p$ consisting of matrices congruent to the identity modulo $p^{m(p)}$, where $M=\prod p^{m(p)}$.

For $\lambda>0$, let $$N^{K(M)}_{\mathrm{sd}}(\lambda)=\sum_{\mathrm{sd,\ }\lambda(\Pi)\leq\lambda} \dim\Pi^{K(M)},$$ the sum ranging over self-dual cuspidal automorphic representations $\Pi$ of $GL_N(\mathbb{A_Q})$ with Laplacian eigenvalue $\leq\lambda$. 

Then there are positive constants $c_1$ and $c_2$ (depending on $K$, but not on $\lambda$) such that, for sufficiently large $\lambda$, one has
$$c_1\lambda^{d/2}\leq N^{K(M)}_{\mathrm{sd}}(\lambda)\leq c_2\lambda^{d/2},$$
where $d=n^2+n$ for $N=2n+\varepsilon$ with $\varepsilon=0, 1$.
\end{theorem}

On the other hand, by Weyl's Law (Theorem \ref{weyl law}) the number of all cuspidal automorphic representations of $GL_N(\mathbb {A_Q})$ (counted similarly) is asymptotically $c\lambda^{D/2}$ for $D=(N^2+N-2)/2$. Hence when $N\neq 2$, the density of self-dual cusp forms is indeed zero.
When $N=2$, self-dual cusp forms have positive density due to the fact that $SO_3=PGL_2$ -- the lifts from this group to $GL_2$ provide for the positive proportion of self-dual representations.


To prove Theorem \ref{intro main}, we prove upper and lower bounds as
Theorems \ref{upper upper} and \ref{lower lower}.
The general idea of the proof is to consider the descent $\pi$  of each self-dual cuspidal automorphic representation $\Pi$ to one of the quasisplit classical groups $\mathbf G(\mathbb A_\mathbb Q)$ and to use results towards Weyl's Law on $\mathbf G(\mathbb A_\mathbb Q)$.
For that we need to relate the relevant properties of $\Pi$ and $\pi$, mainly the Laplacian eigenvalue and the existence of $K$-fixed vectors. 

The Laplacian eigenvalue of an automorphic representation $\sigma$ is determined by the infinitesimal character $\mu$ of the real component $\sigma_\infty$,
namely, by evaluating $\mu$ on the Casimir element $\Omega$.
For real groups we have the Langlands reciprocity and there is an explicit description of the infinitesimal character of $\sigma_\infty$ in terms of the
associated parameter $\varphi: W_\mathbb R\rightarrow$ $^LG$ (here $W_\mathbb R$ is the Weil group and $^LG$ is the Langlands dual group, defined in Section \ref{backgr funct}). In our setting of functoriality from classical groups to $GL_N$, 
the parameter of the functorial lift is just obtained as the composition 
$$\iota\circ\varphi: W_\mathbb R\rightarrow\ ^LG\hookrightarrow\ ^L GL_N.$$
This enables us to conclude that the infinitesimal character, and hence also the Laplacian eigenvalue, are essentially preserved under functoriality, as proved in Theorem \ref{main laplace}.

We also need to relate the existence of $K$-fixed vectors for $\pi$ and $\Pi$.
This is a purely local question, but it is still quite involved, 
as we need to use the local notion of depth to obtain Proposition \ref{upper K fixed}. We shall discuss this in more detail shortly, but first let's explain how we use these results
to prove Theorem \ref{intro main}, starting with the upper bound.

\

Let $K=K(M)$ be a principal congruence subgroup as above and let $\lambda>0$.
We need to estimate
$N^{K}_{\mathrm{sd}}(\lambda)=\sum_{\mathrm{sd,\ }\lambda(\Pi)\leq\lambda} \dim\Pi^{K},$ the sum ranging over self-dual cuspidal automorphic representations $\Pi$ of $GL_N(\mathbb{A_Q})$ with Laplacian eigenvalue $\leq\lambda$.
Each of these self-dual $\Pi$ descends to one of the 
split or non-split symplectic or special orthogonal groups, determined  by the poles of the symmetric square and exterior square $L$-functions 
and by the central character $\chi$ in the case of the quasisplit even orthogonal groups. The restriction of $\chi$ to the compact subgroup $K$ has to be trivial,
and so there are only finitely many possibilities for $\chi$ and (by class field theory) thus only finitely many possibilities for the
corresponding quasisplit group $SO_{2n}^*$ (see Lemma \ref{upper sk}).

Thus we can express $N^K_{\mathrm{sd}}(\lambda)$ as a finite sum $\sum_{G} N^K_{\mathrm{sd}, \mathbf G}(\lambda)$ over different classical groups to which $\Pi$ can descend. 
By \ref{main laplace} and \ref{upper K fixed}, we have a corresponding bound on the Laplacians and compact subgroups 
for the descents $\pi$ on each of these groups $G$, and so we can estimate $N^K_{\mathrm{sd}, \mathbf G}(\lambda)$ from above by the number of all cuspidal automorphic representations
on $G$ as in Weyl's Law. It is not yet known on classical groups, but (after translating from the adelic to classical language)
Donnelly's result \cite{donnelly} gives an upper bound, which is all we need.

The idea of the proof of the lower bound is similar. It suffices to consider only the case when $K$ is a maximal compact subgroup at all places, which somewhat simplifies the situation, but we have to overcome two additional difficulties: First, a cuspidal automorphic representation 
on a classical group need not lift to a cuspidal representation on $GL_N$. 
However, we show that asymptotically these representations have density zero
(Proposition \ref{lower nc}). And second, we need to obtain an estimate 
on the size of ($K$-fixed elements of) global packets. In the case when $K$ is maximal at all places, it in fact follows from Arthur's work that each global packet has at most one $K$-spherical element (Lemma \ref{arthur lemma}).
The lower bound then follows from Weak Weyl's Law for $\mathbf G(\mathbb {A_Q})$
(see Theorem \ref{weak weyl} and the discussion surrounding it).

In the proof of the lower bound, we are relying on Arthur's results \cite{arthur}, which are conditional on the stabilization of the twisted trace formula.
To prove the upper bound, it suffices to consider only functoriality for generic representations, which is available by \cite{cps} unconditionally.

Note that when $N=2n$ and $K$ is a maximal compact subgroup, Weyl's Law is known for $\mathbf G(\mathbb{A_Q})=SO_{2n+1}(\mathbb{A_Q})$ \cite{lv} and we can prove the exact asymptotics (without an error term) for $N^K_{\mathrm{sd}}(\lambda)$ (Corollary \ref{lower cor}).
It should be possible to obtain this result also when $N$ is odd -- but extending this to the case when $K$ need not be maximal seems to be difficult.
One can try to study and estimate the twisted trace formula in a similar way as in the proof of Weyl's Law for $GL_N$ by Lapid-M\" uller \cite{lm}.
This is currently being pursued by Abhishek Parab under the supervision of Professor Shahidi.

\

Let us now turn to the existence of vectors fixed by open compact subgroups in the local setting. We need these results in the proof of a local result, Corollary \ref{depth cor}, which then implies a global statement, Proposition \ref{upper K fixed}.

Now take $F$ to be a $p$-adic field, $\mathbf G$ a connected
reductive group defined over $F$ and $\pi$ an irreducible admissible representation of $G=\mathbf G(F)$.
In the cases when $\mathbf G$ is $GL_N$ or a classical group $SO_N, Sp_{2n}$ or $SO^*_{2n}$, the local Langlands correspondence has been proved by
Harris--Taylor \cite{ht}, Henniart \cite{he}, and Arthur \cite{arthur} (in a slightly weaker form for even orthogonal groups).
Thus we can associate to $\pi$ a representation of the Weil-Deligne group $\varphi: W_F^\prime\rightarrow$ $^LG$.

Fix an open compact subgroup $K$ of $G$. How can we describe the existence of $K$-fixed vectors in $\pi$ in terms of properties of $\varphi$?
(This would enable us to see what happens to the fixed vectors under functoriality, as in Corollary \ref{depth cor}.)

For $\mathbf G=GL_N$ the answer is due to Jacquet, Piatetski-Shapiro, and Shalika \cite{jps}. They define a filtration of $GL_N(F)$ 
by certain ``mirabolic" subgroups
$M_j$ and show that when $\pi$ is generic, the smallest $j$ such that $\pi^{M_j}\neq\{0\}$ is the analytic conductor of $\pi$ given 
by the associated $\varepsilon$-factor.
The $L$-functions and $\varepsilon$-factors are preserved under local Langlands correspondence, and so this provides an answer to our question.

The case of a classical group $G$ is complicated by two facts: first of all, not every tempered representation is generic. But even more seriously, 
we do not have a suitable analogue of the mirabolic subgroups. A more convenient notion is that of depth, defined by Moy and Prasad \cite{mp1}, \cite{mp2} via filtration
by subgroups of $G$ attached to points in the Bruhat-Tits building. The definition itself is somewhat technical (see Section \ref{depth-backgr}), but the main properties of 
depth are that the depth of a representation $\pi$ is a non-negative rational number and that 
depth is preserved under subquotients of parabolic induction and the Jacquet functor. That is, if $M$ is a Levi subgroup of $G$ and 
$\sigma$ an irreducible admissible representation of $M$, then the depth of every irreducible subquotient of the normalized parabolic induction $\mathrm{Ind}_{MN}^{G} \sigma$ is equal
to the depth of $\sigma$ (and similarly for the Jacquet functor).

Finally, we expect to have depth preservation: define the depth of a Weil-Deligne 
representation $\varphi: W_F^\prime\rightarrow$ $^LG$ as $\inf\{r\ |\ \varphi(I_s)=1$ for all $s>r \}$, where $I_s$ is the filtration 
of the inertia subgroup $I$ of $W_F\subset W_F^\prime$. Then $\depth(\pi)$ is expected to be equal to $\depth(\varphi)$. J.-K. Yu \cite{yuottawa} has proved this 
for $GL_N$ and for tamely ramified tori; however, it is not true for example for the non-split
torus $SO_2^*(\mathbb Q_p)$ corresponding to a quadratic extension $F/\mathbb Q_p$ which is ramified at $p=2$ (see Corollary \ref{depth non-pres}).
The fact that depth need not be preserved when residual characteristic is small has already been observed by Gross-Reeder \cite{gross-reeder};
Reeder-Yu \cite{reeder-yu} recently provided more examples of this phenomenon.

We use the notion of depth to relate the existence of fixed vectors for a self-dual representation $\Pi$ of $GL_N(F)$ and its descent $\pi$ to
a classical group $G$. For that, we prove Proposition \ref{rs}, which provides the result needed in counting self-dual representations.

We also look at proving the depth preservation for classical groups. 
As we have noted before, 
it is not true in general for non-split even orthogonal groups, but we prove at least one inequality for generic representations in Theorem \ref{descent-bound}.
The proof is based on the  local descent construction of Jiang and Soudry \cite{js}. Since depth is preserved under parabolic induction,
we can assume that $\Pi$ and $\pi$ are supercuspidal, in which case the local descent gives an explicit way of constructing $\pi$ from $\Pi$
in terms of certain inductions and Jacquet quotients, which preserve depth. The obstruction to proving depth preservation
by this method is that one of the steps in the construction consists of taking the restriction of a representation to a subgroup, which may decrease the depth. See Section \ref{depth-descent} for the proof and details about local descent (in the case of $\mathbf G=SO_{2n+1}$).

The restriction that appears in the local descent to $SO_{2n+1}$ is that of a representation of $SO_{2n+2}(F)\times (GL_1(F))^{n-1}$ to the first component $SO_{2n+2}(F)$ and then to $SO_{2n+1}(F)\subset SO_{2n+2}(F)$.
In Section \ref{depth restr} we study a part of this problem, i.e., the restriction from $SO_{2n+2}(F)$ to $SO_{2n+1}(F)$
and show that depth does not decrease under it. To prove this, we use some results on irreducibility of parabolically induced representations which follow from the Langlands-Shahidi method and are stated in Section \ref{backgr induced}.
We hope to work on this problem further in the future and hopefully obtain full depth preservation for quasisplit classical groups at least in some cases.

\

Our method should work quite generally for obtaining upper bounds on the number of functorial lifts from $G(\mathbb {A_Q})$ to $GL_N(\mathbb {A_Q})$ whenever functoriality is known at least for generic representations. 
This is the case for lifts from unitary and general spin groups to $GL_N$ (\cite{mok}, \cite{ash}); 
the image of functoriality is formed by conjugate self-dual and essentially self-dual representations, respectively. 
As was the case in this thesis, proving lower bounds is much more delicate and requires more detailed information about the image of functoriality and the global packets. This seems to be available at least in the case of unitary groups.
We plan to undertake this research in more detail in the future.
It also seems possible to extend these results to other number fields besides $\mathbb Q$, with the caveat that Weyl's Law is not known in most of these cases.

\

One may wonder how our results on the number of self-dual automorphic representations relate to Rohrlich's results \cite{rohrlich} concerning Galois representations.
An automorphic representation attached to a Galois representation will have Laplacian eigenvalue 0; essentially, Rohrlich is considering the asymptotics when the Laplacian eigenvalue (which corresponds to the weight in classical setting) is fixed (and small) and the level of the compact subgroup goes to infinity. On the other hand, our results concern the asymptotics when the level is fixed and the Laplacian eigenvalue goes to infinity. This also explains the seeming disparity between our result that self-dual cusp forms on $GL_2$ have positive density, whereas 2-dimensional self-dual Galois representations have density zero by Rohrlich.

Unifying these results or obtaining asymptotical results on the automorphic side with fixed eigenvalue and increasing level, seems to be an interesting and hard problem.

%
%

%
%
%

\chapter{Background}

In this chapter we collect various background information on functoriality, results related to Weyl's Law, and some consequences of the Langlands-Shahidi method for irreducibility of induced representations that we will need (in Sections \ref{backgr funct}, \ref{backgr weyl}, and \ref{backgr induced}, respectively). 
Note that further background on Laplacian eigenvalues and depth is in their respective chapters.

\section{Basic Definitions and Notation}\label{backgr basic}

To fix the notation, let us review the definitions of
the following reductive algebraic groups defined over $\mathbb Q$ with which we shall be mostly dealing. 
Let $I_n$ be the $n\times n$ identity matrix and $J_n$ the $n\times n$ matrix with 1's along the second diagonal.

\

$GL_N=\{g| g$ is invertible$\}$

$SL_N=\{g| \det g=1\}$

$SO_{2n+1}=\{g\in SL_{2n+1}|gM ^tg=M\}$, where $M=\left( \begin{array}{ccc} 0 & 0 & I_n \\ 0 & 1 & 0 \\ I_n & 0 & 0 \end{array} \right)$

$Sp_{2n}=\{g\in SL_{2n}|gM ^tg=M\}$, where $M=\left( \begin{array}{cc} 0 & I_n \\ -I_n & 0 \end{array} \right)$

$SO_{2n}=\{g\in SL_{2n}|gM ^tg=M\}$, where $M=\left( \begin{array}{cc} 0 & I_n \\ I_n & 0 \end{array} \right)$

$SO^\ast_{2n}(F)=SO^{\ast, \tau}_{2n}(F)=\{g\in SL_{2n}(F)|gM ^tg=M\}$ for a field $F\supset \mathbb Q$ and $\tau\in F$ such that $-\tau$ is not a square in $F$, where $M=M_\tau=\left( \begin{array}{ccc} 0 & 0 & J_{n-1} \\ 0 & \Lambda_\tau & 0 \\ J_{n-1} & 0 & 0 \end{array} \right)$ and 
$\Lambda_\tau=\left(\begin{array}{cc} 1 & 0 \\ 0 & \tau \end{array} \right)$

The group $SO^{\ast, \tau}_{2n}(F)$ is attached to a quadratic extension $F(\sqrt{-\tau})/F$, to which class field theory associates a quadratic character $\eta_{F(\sqrt{-\tau})/F}$.

\

All of the groups above are split, except for $SO^{\ast, \tau}_{2n}$, which is quasisplit but not split (when $-\tau$ is not a square).

Throughout, by a $\textit{quasisplit classical group}$, we shall mean one of the groups $SO_{N},$ $SO^\ast_{2n},$ $Sp_{2n}$ defined above.

\

For a number field $F$, we denote by $\mathbb{A}_F$ the ring of adeles of F. When $F=\mathbb Q$, we often just write $\mathbb A$ for $\mathbb{A_Q}$. Denote the ring of integers of $\mathbb Q_p$ by $\mathcal O_p$.

\

We shall generally denote algebraic groups by boldface letters $\mathbf {G, H, M}$ and their $F$-points by $G, H, M$, etc.

Throughout the thesis, $\mathrm{Ind}_H^G \sigma$ denotes the non-normalized induction (unless explicitly stated otherwise). In Section \ref{backgr induced} we shall define the normalized (twisted) parabolic induction $I(\nu, \sigma)$.

\section{Langlands Functoriality and Reciprocity}\label{backgr funct}

Let us now review some theorems about Langlands correspondence and functoriality for $GL_N$ and classical groups.
We shall mostly follow \cite{arthur} and \cite{cps}.

Let $F$ be a local or global field of characteristic zero. Let $W_F$ be the Weil group, 
and when $F$ is non-archimedean, let $W_F^\prime=W_F\times SU(2)$ be the Weil-Deligne group. 
To provide a unified treatment of Langlands correspondence in the archimedean and $p$-adic cases, 
denote by $L_F$ the local Langlands group defined as $W_F$ when $F$ is archimedean and 
$W_F^\prime$ when $F$ is non-archimedean.

Let $\mathbf G$ be a connected reductive group over $F$. 
Let $G^\vee$ be the (complex) Langlands dual group and $^LG=G^\vee\rtimes W_F$ the $L$-group of $G$.

\

Let us first take $F$ to be local.
Denote by $\Phi(G)$ the set of $G^\vee$-orbits of semisimple continuous $L$-homomorphisms $\varphi: L_F\rightarrow$ $^LG$. 
Let $\Pi(G)$ be the set of equivalence classes of irreducible admissible representations $\pi$ of $G=\mathbf G(F)$.

The local Langlands correspondence then conjectures that there is a reciprocity map from $\Pi(G)$ to $\Phi(G)$ with various natural properties, e.g., that it preserves  $L$-functions. The fibers of the reciprocity map are called $L$-packets and are expected to be finite.

Suppose now that we have two reductive groups $\mathbf G_1$ and $\mathbf G_2$ and
a homomorphism $f:$ $^LG_1\rightarrow$ $^LG_2$.
Each (irreducible admissible) representation $\pi_1$ of $G_1$ is expected to correspond to some $\varphi: L_F\rightarrow$ $^L G_1$. Composing $\varphi$ with $f$ we obtain $f\circ\varphi: L_F\rightarrow$ $^L G_2$, to which should correspond a representation $\pi_2$ of $G_2$.
This gives us a conjectural
map from $\Pi(G_1)$ to $\Pi(G_2)$, referred to as local Langlands functoriality.

Local Langlands correspondence and functoriality have been proved by Langlands \cite{langl} when $F=\mathbb R$ for all groups. 
When $F$ is non-archimedean and $\mathbf G=GL_N$, reciprocity is due to Harris-Taylor \cite{ht} and Henniart \cite{he}.
Assuming the stabilization of the twisted trace formula for $GL_N$, Arthur \cite{arthur} has recently proved reciprocity for classical groups $\mathbf G$ 
(only up to outer conjugacy by $O_{2n}$ in the case of even orthogonal groups) and functoriality corresponding to natural (endoscopic) embeddings $^LG\hookrightarrow$ $^LGL_N$. We shall discuss his results in more detail at the end of this section.

\

Take now a global field $F$. It is not known what extension $L_F$ of the Weil group we need to take in order to expect Langlands reciprocity between representations $\varphi: L_F\rightarrow$ $^L G$ and automorphic representations of $\mathbf G(\mathbb A_F)$ to hold.
But we still can consider conjectural global functoriality between automorphic representations of $\mathbf G_1(\mathbb A_F)$ and $\mathbf G_2(\mathbb A_F)$ as above.

Functoriality from classical groups to $GL_N$ has been established by Cogdell, Kim, Piatetski-Shapiro, and Shahidi \cite{ckpss}, \cite{cps} for generic automorphic representations. A precise description of the image of functoriality is given by the descent method of Ginzburg, Rallis, and Soudry \cite{grs}:

\begin{theorem}\label{functoriality} (\cite{cps}, \cite{grs})
Let $\mathbf G$ be a quasisplit classical group with an embedding $^LG\hookrightarrow$ $^LGL_N$
and let $\chi_G$ and $R$ be as in Table \ref{table funct}.

Let $\pi$ be a globally generic cuspidal automorphic representation of $\mathbf G(\mathbb A_F)$. Then the functorial lift of $\pi$ to an automorphic representation $\Pi$ of $GL_N(\mathbb A_F)$ is self-dual with central character $\omega_\Pi=\chi_G$.
$\Pi$ is the normalized parabolic induction $\Pi=\mathrm{Ind}(\Pi_1\otimes\cdots\otimes\Pi_k)=\Pi_1\boxplus\cdots\boxplus\Pi_k$,
where each $\Pi_i$ is a self-dual cuspidal automorphic representation of $GL_{N_i}(\mathbb A_F)$ such that the partial $L$-function $L^T(s, \Pi_i, R)$
(with $T$ a sufficiently large set of places of $F$ containing all archimedean ones)
has a pole at $s=1$ and $\Pi_i\not\simeq\Pi_j$ for $i\neq j$. Moreover, every such $\Pi$ is a functorial lift of some $\pi$.
\end{theorem}

\begin{table}[h]\caption{Functoriality}\label{table funct}
\begin{center}
\begin{tabular}{|c|c|c|c|c|}
\hline
$\mathbf G$ & $\mathbf G^\vee$ & $N$ & $R$ & $\chi_G$ \\
\hline
$Sp_{2n}$ & $SO_{2n+1}$ & $2n+1$ & $\mathrm{Sym}^2$ & $\mathbf 1$  \\
$SO_{2n+1}$ & $Sp_{2n}$ & $2n$ & $\bigwedge^2$ & $\mathbf 1$  \\
$SO_{2n}$ & $SO_{2n}$ & $2n$ & $\mathrm{Sym}^2$ & $\mathbf 1$  \\
$SO^\ast_{2n}$ & $SO_{2n}$ & $2n$ & $\mathrm{Sym}^2$ & $\eta_{F(\sqrt{-\tau})/F}$  \\
\hline
\end{tabular}
\end{center}
\end{table}

For all automorphic representations of classical groups, not only the generic ones, global functoriality was proved by Arthur \cite{arthur} (again assuming the stabilization of the twisted trace formula for $GL_N$ and only up to outer conjugacy by $O_{2n}$ for even orthogonal groups). 

Take $\mathbf G$ to be a quasisplit classical group and let us now briefly review some of Arthur's results which we shall use, starting again with the local case:

Let $\widetilde\Psi^+(G)$ the set of (not necessarily irreducible) $L$-homomorphisms $\psi: L_F\times SU(2)\rightarrow$ $^LG$,
taken up to conjugacy by $\mathrm{Out}(G^\vee)$. Arthur then defines subsets 
$\widetilde\Psi_{\mathrm{sim}}(G)\subset\widetilde\Psi_2(G)\subset\widetilde\Psi(G)\subset \widetilde\Psi_{\mathrm{unit}}^+(G)\subset\widetilde\Psi^+(G)$ and 
$\widetilde\Phi_{\mathrm{bdd}}(G)$. Let $\widetilde\Pi_{\mathrm{unit}}(G)$ be the set of 
of unitary irreducible admissible representations of $G$ (taken up to conjugacy by outer automorphisms of $G$) and $\widetilde\Pi_{\mathrm{temp}}(G)$ the set of those which are tempered. 
Finally, for $\psi\in \widetilde\Psi(G)$ we have the group $\mathcal S_\psi$. See Chapter 1 of \cite{arthur} for precise definitions and details.

We then have the local classification of Theorem 1.5.1 in \cite{arthur}.

\begin{theorem}\label{arthur 151}(\cite{arthur}, 1.5.1)
a) For every $\psi\in\widetilde\Psi(G)$, there is a finite subset $\widetilde{\Pi}_\psi$ of $\widetilde\Pi_{\mathrm{unit}}(G)$
equipped with a canonical mapping 
$$\pi\mapsto \langle\cdot, \pi\rangle, \ \ \ \ \pi\in\widetilde{\Pi}_\psi,$$
from $\widetilde{\Pi}_\psi$ to the group $\hat{\mathcal S}_\psi$ of characters on $\mathcal S_\psi$ such that
$\langle\cdot, \pi\rangle=1$ if $G$ and $\pi$ are unramified (relative to a suitable fixed compact open subgroup $K$).

b) If $\varphi=\psi$ belongs to the subset $\widetilde\Phi_{\mathrm{bdd}}(G)$ of parameters in $\widetilde\Psi(G)$ which are trivial on the factor $SU(2)$,
the elements in $\widetilde{\Pi}_\varphi$ are tempered and the corresponding mapping from $\widetilde{\Pi}_\phi$ to $\hat{\mathcal S}_\psi$ is injective.
Moreover, every element in $\widetilde\Pi_{\mathrm{temp}} (G)$ belongs to exactly one packet $\widetilde{\Pi}_\phi$. Finally, if $F$ is non-archimedean, 
the mapping from $\widetilde{\Pi}_\varphi$ to $\hat{\mathcal S}_\psi$ is bijective.
\end{theorem}

To a parameter $\psi\in \widetilde\Psi^+(G)$ one can associate a parameter $\varphi_\psi\in\widetilde\Phi(G)$, given by
$$\varphi_\psi(w)=\psi(w, \left( \begin{array}{cc} |w|^{1/2} & 0 \\ 0 & |w|^{-1/2} \end{array} \right).$$

Then $\widetilde{\Pi}_{\varphi_\psi}\subset\widetilde{\Pi}_{\psi}$ and for $K$-spherical representations we have the following proposition:

\begin{proposition}\label{arthur prop}
Let $\mathbf G$ be a split orthogonal or symplectic group and $K$ a fixed maximal compact subgroup of $\mathbf G(F)$; if $F$ is non-archimedean we take $K=\mathbf G(\mathcal O_{F})$.

Let $\psi\in\widetilde{\Psi}^+_{\mathrm{unit}}(G)$ and assume that $\pi\in\widetilde{\Pi}_{\psi}$ has a non-zero $K$-fixed vector. 
Then $\pi\in\widetilde{\Pi}_{\varphi_\psi}$. Moreover, there is at most one such a representation $\pi$ in $\widetilde{\Pi}_{\psi}$.
\end{proposition}

Let us note that similar results were discussed by Shahidi \cite{shahidi-kyoto}.

\

Let us now discuss the global results, and so assume that $F$ is global. We denote by $\mathcal A(G)$, $\mathcal A_{2}(G)$ and $A_{\mathrm{cusp}}(G)$ 
the sets of (unitary) automorphic representations of $\mathbf G(\mathbb{A}_F)$, respectively those that are in the discrete spectrum or cuspidal.

Let $\Psi_{\mathrm{sim}}(N)=\Psi_{\mathrm{sim}}(GL(N))$ be the set of formal tensor products $\psi=\mu\boxtimes\nu$, where $N=mn$, $\mu\in\mathcal A_{\mathrm{cusp}}(GL(m))$ and
$\nu$ is the unique irreducible representation of $SU(2)$ of degree $n$. The parameter $\psi$ corresponds to an isobaric sum
$$\Pi=\mu\left(\frac{n-1}2\right)\boxplus \mu\left(\frac{n-3}2\right)\boxplus\cdots\boxplus\mu\left(-\frac{n-1}2\right),$$
where $\mu(i): x\mapsto \mu(x)|\det x|^i$. By a Theorem of Moeglin-Waldspurger, the set $\Psi_{\mathrm{sim}}(N)$ parametrizes the discrete spectrum 
$\mathcal A_2(GL(N))$.

Let $\widetilde\Psi_{\mathrm{sim}}(N)$ be the subset of self-dual elements in $\Psi_{\mathrm{sim}}(N)$ and let $\widetilde\Psi_{2}(N)$ be the set of formal unordered sums
$\psi=\psi_1\boxplus\cdots\boxplus\psi_r$ for distinct elements $\psi_i\in \widetilde\Psi_{\mathrm{sim}}(N_r)$, where $N=N_1+\cdots+N_r$.
Elements of $\widetilde\Psi_{2}(N)$ thus correspond to isobaric sums $\Pi_1\boxplus\cdots\boxplus\Pi_r$ for $\Pi_i\in \mathcal A_2(GL(N_i))$.

Finally, let $\widetilde\Psi_{2}(G)$ be the set of those parameters $\psi\in\widetilde\Psi_{2}(N)$,
which are expected to be lifts of automorphic representations from $\mathbf G(\mathbb{A}_F)$ (see \cite{arthur}, Section 1.4 for a precise definition).

We have a localization mapping $\psi\mapsto\psi_v$ from $\widetilde\Psi_{2}(G)$ to 
$\widetilde{\Psi}_{\mathrm{unit}}^+(G_v)$ and so we can define the global packet 
$$\widetilde\Pi_\psi=\left\{
\bigotimes_v\pi_v | \pi_v\in\widetilde\Pi_{\psi_v}, \langle \cdot, \pi_v\rangle=1 \mathrm{\ for\ almost\ all\ } v
\right\}.$$

Let $\widetilde{\mathcal H}(G)$ be the product of the local symmetric Hecke algebras.

\begin{theorem}\label{arthur 152} (\cite{arthur}, 1.5.2)
There is an $\widetilde{\mathcal H}(G)$-module isomorphism
$$L^2_\mathrm{disc}\simeq \bigoplus_{\psi\in\widetilde\Psi_{2}(G)}\bigoplus_ {\pi\in\widetilde\Pi_{\psi}(\varepsilon_\psi)} m_\psi\pi, $$
where $m_\psi$ equals 1 or 2, $\varepsilon_\psi: \mathcal S_\psi\rightarrow \{\pm 1\}$
is an explicitly defined character and $\widetilde\Pi_{\psi}(\varepsilon_\psi)$ is the subset
of elements of $\widetilde\Pi_{\psi}$ such that $ \langle \cdot, \pi\rangle=\varepsilon_\psi$.
\end{theorem}

When $G$ is not an even orthogonal group, $m_\psi$ is always 1.

Finally, let us consider $K$-spherical representations in an A-packet $\widetilde\Pi_{\psi}$ 
for $K=K_\infty\times \prod K_p$ with $K_p=\mathbf G(\mathcal O_p)$ and a fixed maximal compact subgroup of $\mathbf G(\mathbb R)$ (to simplify notation a little, we take $\mathbf G$ split and $F=\mathbb Q$ here).

\begin{lemma}\label{arthur lemma}
Let $\mathbf G$ be a split classical group, $F=\mathbb Q$ and $\psi\in\widetilde{\Psi}_2(G)$. Then there is at most one representation
$\pi\in \widetilde\Pi_{\psi}$ such that $\pi^K\neq 0$
(where $K$ is the maximal compact subgroup as above).
\end{lemma}

\begin{proof}
Let $\pi=\otimes \pi_v$. Each $\pi_v$ is $K_v$-spherical, and so by Proposition \ref{arthur prop}, there is at most one such $\pi_v$. Hence there is at most one global representation $\pi$.
\end{proof}


\

Note that we need general Langlands correspondence and functoriality from \cite{arthur} only in the last chapter where we prove a lower bound on the number of self-dual representations. In the rest of the thesis we could just work with the generic functoriality (which is available unconditionally). The analogue of Proposition \ref{rs} is in fact easier to prove only for generic representations than in general.

\section{Weyl's Law}\label{backgr weyl}

Let $\mathbf G$ be a reductive group over $\mathbb Q$. For a prime $p$, take a compact open subgroup $K_p$ of $\mathbf G(\mathbb Q_p)$,
also let $K_\infty$ be a maximal compact subgroup of $\mathbf G(\mathbb R)$.
This gives us a compact subgroup $K=K_\infty\times \prod K_p$ of $\mathbf G(\mathbb A)$. Finally, denote by $A_\mathbf G$ the maximal split torus in the center of $\mathbf G$ and let $A_\mathbb R=A_\mathbf G(\mathbb R)^0$ be the connected component of the $\mathbb R$-points of $A_\mathbf G$.

Weyl's Law conjecturally gives asymptotic results for the number of cuspidal automorphic representations of $\mathbf G(\mathbb A)$. 
For $\lambda>0$, define $N_G^K(\lambda)=\sum_{\lambda(\pi)\leq\lambda} \dim\pi^K$, where the sum is over
all cuspidal automorphic representations $\pi=\otimes\pi_v$ of $\mathbf G(\mathbb A)$ such that the
Laplacian eigenvalue $\lambda(\pi)$ of the restriction of $\pi_\infty$ to $\mathbf G^1(\mathbb R)$ is at most $\lambda$
(see Section \ref{laplace infin} for the definition of the Laplacian eigenvalue). 
Here $\mathbf G^1=\{g\in \mathbf G|\ |\chi(g)|=1$\ for all $\chi\in X^\ast(\mathbf G)\}$.

Weyl's Law then asserts that $N_G^K(\lambda)=c\lambda^{d/2}+o(\lambda^{d/2})$ as $\lambda\rightarrow\infty$ for
an explicit constant $c$, where $d$ is the dimension of the locally symmetric space 
$$X=X_G^K=A_{\mathbb R}\mathbf G(\mathbb Q)\backslash\mathbf G(\mathbb A)/K.$$

In this form, Weyl's Law was proved for $\mathbf G=GL_N$ by M\" uller \cite{m} and then by Lapid-M\" uller \cite{lm} with an improved error term:

\begin{theorem}\label{weyl law} \cite{m}
Let $\mathbf G=GL_N$, let $K_p=K_p(m)$ be the principal congruence subgroup of level $m$ (i.e., the subgroup consisting of matrices congruent to the identity modulo $p^m$) and $K_\infty=O(N)$.
Then
$$N_{GL(N)}^K(\lambda)=\frac{\mathrm{vol}(X)}{(4\pi)^{d/2}\mathbf{\Gamma}(d/2+1)}\lambda^{d/2}+o(\lambda^{d/2}),$$ where $\mathbf{\Gamma}$ is the standard gamma function.
\end{theorem}

When $\mathbf G$ is a split adjoint semisimple group, Weyl's Law was proved by Linden\-strauss-Venkatesh \cite{lv}. 
It seems that a similar proof should work for a general split semisimple group, but
this has not appeared in literature yet. At the very least, one has weak Weyl's Law:

\begin{theorem}\label{weak weyl}
There are constants $c_1$ and $c_2$ such that for all sufficiently large $\lambda$,
one has 
$c_1\lambda^{d/2}<N_G^K(\lambda)<c_2\lambda^{d/2}$.
\end{theorem}

When $\mathbf G$ is a split simply connected semisimple group, Theorem \ref{weak weyl} was proved by Labesse-M\" uller \cite{lab-m}.


On a general group the full Weyl's Law hasn't been proved yet, but we have the following upper bound due to Donnelly \cite{donnelly}.

Let us first review the definition of an arithmetic subgroup: A subgroup $H$ of a Lie group $G$ is arithmetic if there exists $k\in\mathbb N$ and a faithful representation $\rho: G\rightarrow GL_k(\mathbb C)$ such that $\rho(H)$ is commensurable with $\rho(G)\cap GL_k(\mathbb Z)$. This condition does not depend on the choice of $k$ and $\rho$.

\begin{theorem}\cite{donnelly}\label{classical donnelly}
Let $G$ be a (non-compact) semisimple Lie group, 
$K$ a maximal compact subgroup of $G$ and $\Gamma$ an arithmetic subgroup of $G$. 
Let $N(\lambda)$ be the number of linearly independent Laplacian eigenfunctions
in $L^2_{cusp}(K\backslash G /\Gamma)$ with eigenvalue less than $\lambda$.
Then 
$$\limsup_{\lambda\rightarrow\infty} \frac{N(\lambda)}{\lambda^{d/2}}\leq
\frac{\mathrm{vol}(K\backslash G /\Gamma)}{(4\pi)^{d/2}\mathbf{\Gamma}(d/2+1)},$$

where $d$ is the dimension of $K\backslash G /\Gamma$.
\end{theorem}

The dimensions of the symmetric spaces are given in Table \ref{table dimensions}.

\begin{table}[h]\caption{Dimensions}\label{table dimensions}
\begin{center}
\begin{tabular}{|c|c|c|c|}
\hline
$G$ & $\dim G$ & $\dim K$ & $d_G=\dim X_G$ \\
\hline
$SL_{2n+1}(\mathbb R)$ & $4n^2+4n$ & $2n^2+n$ & $2n^2+3n$  \\
$SL_{2n}(\mathbb R)$ & $4n^2-1$ & $2n^2-n$ & $2n^2+n-1$ \\
$Sp_{2n}(\mathbb R)$ & $2n^2+n$ & $n^2$ & $n^2+n$  \\
$SO_{2n+1}(\mathbb R)$ & $2n^2+n$ & $n^2$ & $n^2+n$  \\
$SO_{2n}(\mathbb R)$ & $2n^2-n$ & $n^2-n$ & $n^2$  \\
$SO^\ast_{2n}(\mathbb R)$ & $2n^2-n$ & $n^2$ & $n^2-n$  \\
\hline
\end{tabular}
\end{center}
\end{table}

We shall need an adelic reformulation of Donnelly's result:

\begin{proposition}\label{adelic donnelly}
Let $\mathbf G$ be a connected reductive group such that $\mathbf G^1(\mathbb R)$ is not compact. 
Let $K$ be an arithmetic subgroup of $\mathbf G(\mathbb A)$ and 
$X=A_{\mathbb R}\mathbf G(\mathbb Q)\backslash\mathbf G(\mathbb A)/K$.

Then for $\lambda$ large enough we have 
$$N_G^K(\lambda)\leq
\frac{\mathrm{vol}(X)}{(4\pi)^{d/2}\mathbf{\Gamma}(d/2+1)}\lambda^{d/2},$$
where $d$ is the dimension of $X$.
\end{proposition}

Here we extend the definition of an arithmetic subgroup to the adelic setting as follows:
Let $K_f=\prod_{p<\infty} K_p$.
By the strong approximation theorem we have that $A_{\mathbb R}\mathbf G(\mathbb Q)\backslash\mathbf G(\mathbb A)/\mathbf G(\mathbb R)K_f$ is finite.
Denote by $x_1=1, x_2, \dots, x_l$ the double coset representatives in $\mathbf G(\mathbb A_f)$
and let $\Gamma_i=(\mathbf G(\mathbb Q)\cdot x_iK_fx_i^{-1})\cap \mathbf G^1(\mathbb R)$.
We define $K$ to be arithmetic if each of these $\Gamma_i$ is an arithmetic subgroup of $\mathbf G(\mathbb R)$ (as defined before Theorem \ref{classical donnelly}).

\begin{proof}
We argue similarly as on page 134 of \cite{lm}.
Let $\Gamma_i$ be as above. Since $K$ is arithmetic, all of these $\Gamma_i$ are arithmetic subgroups of $\mathbf G(\mathbb R)$ by definition.

We obtain an isomorphism of $\mathbf G(\mathbb R)$-spaces
$$A_{\mathbb R}\mathbf G(\mathbb Q)\backslash\mathbf G(\mathbb A)/K_\infty K_f
\simeq\bigsqcup_i\Gamma_i\backslash \mathbf G^1(\mathbb R)/K_\infty.$$

Automorphic representations correspond to automorphic forms on $\Gamma_i\backslash \mathbf G^1(\mathbb R)/K_\infty$; 
denote by $N_i(\lambda)$ the number of linearly independent Laplacian eigenfunctions
in the space $L^2_{cusp}(\Gamma_i\backslash \mathbf G^1(\mathbb R)/K_\infty)$ with eigenvalue less than $\lambda$. 
Then $N_G^K(\lambda)=\sum_i N_i(\lambda)$.

Using Donnelly's upper bound for each $N_i(\lambda)$, we obtain the desired estimate for $N_G^K(\lambda)$.
\end{proof}

Note that the assumption that $\mathbf G^1(\mathbb R)$ is not compact is satisfied for
all of the groups with which we shall be dealing, except for non-split $SO_2^\ast$.

\section{$L$-Functions and Irreducibility of Induced Representations}\label{backgr induced}

In this Section we review some results concerning irreducibility of parabolically induced representations, which follow from the Langlands-Shahidi method. We shall use these results in Section \ref{depth restr}. Our main reference is \cite{sh}.

Let $F$ be a non-archimedean field of characteristic 0, $\mathbf G$ a quasisplit connected reductive group over $F$ and $G=\mathbf G(F)$.

The theory works more generally as well, but for our purposes, let us assume that $P=MN$ is a standard maximal parabolic subgroup of $G$. Let $A$ be the split component of $M$, $\liea^\ast=X^\ast(M)\otimes \mathbb R$ and $\liea=\mathrm{Hom} (X^\ast(M), \mathbb R)$. 
Define a map $H_M: M\rightarrow\liea$ by $\exp\langle\chi, H_M(m)\rangle=|\chi(m)|_F$.
Denote by $\rho_P$ half of the sum of roots in $\mathrm {Lie}(N)$ and by $\alpha$ the unique simple root of $A$ in $\mathrm {Lie}(N)$. Finally, set 
$\tilde{\alpha}=\langle\rho_P, \alpha \rangle^{-1}\rho_P\in\liea^\ast$.

Let $\sigma$ be an irreducible unitary generic (admissible) representation of $M$. 
For $\nu\in\lieac^\ast$ denote by $I(\nu, \sigma)$ the induced representation 
$\mathrm {Ind}_P^G\ \sigma\otimes\exp\langle\nu+\rho_P, H_M(\cdot) \rangle\otimes\mathbf 1$.

To the parabolic $P$ we can associate representations $r_i$ with highest weight vectors given by the diagrams starting on page 183 of \cite{sh}.
Using the Langlands-Shahidi method, one can then define the local coefficients $C_\chi(s\tilde{\alpha}, \sigma, w_0)$ and $L$-functions $L(s, \sigma, r_i)$.

We have the following theorem:

\begin{theorem}\label{backgr shahidi} (\cite{sh81}, 3.3.1 and \cite{sh}, 8.4.9, 8.5.1)
Assume that $\sigma$ is an irreducible unitary generic supercuspidal representation of $M$. Then
$$C_\chi(s\tilde{\alpha}, \sigma, w_0)\sim 
L(1-s, \tilde{\sigma}, r_1)L(1-2s, \tilde{\sigma}, r_2)/
L(s, {\sigma}, r_1) L(2s, {\sigma}, r_2),$$

where $\sim$ means equality up to a monomial in $q^{-s}$.
The following are equivalent:

a) $L(s, {\sigma}, r_i)$ has a pole at $s=0$ for $i=1$ or 2 and only for one of them.

b) $I(0, \sigma)$ is irreducible and $\sigma$ is ramified, i.e., $w_0(\sigma)\simeq\sigma$.

\

Moreover, assume that these conditions hold and let $i=1$ or 2 be such that $L(s, {\sigma}, r_i)$ has a pole at $s=0$. Then for $s>0$, we have that $I(s\tilde{\alpha}, \sigma)$ is reducible only when $s=1/i$, for all other $s$ it is irreducible.

\

If $\sigma$ is ramified and $I(0, \sigma)$ is reducible, then all $I(s\tilde{\alpha}, \sigma)$ with $s>0$ are irreducible.

If $\sigma$ is unramified, then all $I(s\tilde{\alpha}, \sigma)$ with $s\geq 0$ 
are irreducible.
\end{theorem}

Let us now apply this theorem to prove a lemma which we shall need in Section \ref{depth restr}.

\begin{lemma}\label{lemma induction}
Let $\mathbf G=SO_{2n+1}$ and $P=MN$ be standard maximal parabolic with Levi $\mathbf M=GL_1\times SO_{2n-1}$. Let $\sigma$ be an irreducible unitary generic supercuspidal representation of $SO_{2n-1}(F)$ and $\chi$ a unitary character of $F^\times$ such that $\chi^2\neq \mathbf 1$.
Then the induced representation $\mathrm{Ind}_P^G\ \sigma\otimes\chi$ is irreducible.
\end{lemma}

\begin{proof}
Since $\chi^2\neq \mathbf 1$, we have that $\sigma\otimes\chi$ is unramified in the sense of Theorem \ref{backgr shahidi}. By this theorem we then see that all
$I(s\tilde{\alpha}, \sigma\otimes\chi)$ are irreducible. In particular, $I(\tilde{\alpha}/2, \sigma\otimes\chi)$ is irreducible.

We have
$\alpha=e_1-e_2$ and $\rho_P=e_2/2$. Hence $\langle\rho_P, \alpha \rangle=-1/2$, and so $\tilde{\alpha}=-2\rho_p$. Hence we see that
$I(\tilde{\alpha}/2, \sigma\otimes\chi)=I(-\rho_P, \sigma\otimes\chi)$ is 
the non-normalized induction $\mathrm{Ind}_P^G\ \sigma\otimes\chi$, the irreducibility of which we wanted to prove.
\end{proof}

%
%
%

\chapter{Laplacian Eigenvalues}

In this chapter we shall study the behavior of Laplacian eigenvalues under functoriality.
This is given by Theorem \ref{main laplace}, which we will later use in the proofs in Chapters 5 and 6.

Laplacian eigenvalues are obtained by evaluating the infinitesimal character on the
Casimir element. Hence we start by discussing the Casimir elements in Section \ref{lapl-casimir} and then consider the infinitesimal characters in Section \ref{laplace infin}. Finally, we prove Theorem \ref{main laplace} in Section \ref{laplace proof}.

Our main reference is \cite{knapp}, Chapter V.5. 

Throughout this chapter, all groups and representations are real, unless stated otherwise.

\section{Casimir Element}\label{lapl-casimir}

In this Section, we review the definition of the Killing form and Casimir element, following \cite{knapp} (mostly Chapter V).

Let $\lieg$ be a Lie algebra (viewed as a Lie subalgebra of $\mathrm{Mat}_{m\times m}(\mathbb C)$ for some $m$).
For $X, Y\in\lieg$, we define the Killing form as $B(X, Y)=\Tr(\mathrm{ad\ } X\ \mathrm{ad\ } Y)$. It is a symmetric bilinear form on $\lieg$. We also have the bilinear form $C(X, Y)=\Tr(XY)$.
When $\lieg$ is a simple complex Lie algebra, $B(X, Y)=b_\lieg C(X, Y)$ for some constant $b_\lieg$. 

We shall be dealing mostly with the Lie algebras $\liesl_N(\mathbb C)$, $\liesp_{2n}(\mathbb C)$, $\lieso_N(\mathbb C)$, and $\lieso_{2n}^\ast(\mathbb C)$.
For each of them, the Killing form is a multiple of $C(X, Y)$ and the corresponding constants are 
$b_{\liesl(N)}=2N$, $b_{\liesp(2n)}=2n+2$, and $b_{\lieso(N)}=b_{\lieso^\ast(N)}=N-2$ (see eg. \cite{fh}, Ex. 14.36).

Now, let $\lieg$ be a real Lie algebra and $\liegc=\lieg\otimes_\mathbb R\mathbb C$ its complexification.
Let $\lieh$ be a Cartan subalgebra of $\lieg$ (and $\liehc$ its complexification).
Let $\Phi$ be the set of roots of $\liegc$ with respect to $\liehc$.

Choose an orthonormal basis $H_1, \dots, H_n$ of $\liehc$ with respect to the Killing form $B$ on $\liegc$. Also choose root vectors $E_\alpha$ so that $B(E_\alpha, E_{-\alpha})=1$ for all $\alpha\in\Phi$. Then the Casimir element of $\liegc$ is
$\Omega=\sum_{i=1}^n H_i^2 + \sum_{\alpha\in\Phi} E_\alpha E_{-\alpha}$ as an element of
the center $Z(\liegc)$ of the universal enveloping algebra.

We have the Harish-Chandra isomorphism $\gamma: Z(\liegc)\rightarrow \mathcal H^W$, where
$\mathcal H$ is the universal enveloping algebra of $\liehc$ and $\mathcal H^W$ are its 
Weyl group invariants. For more details, see \cite{knapp}, Ch. V.5.

To compute the Laplacian eigenvalues of our representations, we shall need the images of the Casimir element under the Harish-Chandra isomorphism. In general,
it is $\gamma(\Omega)=\sum H_i^2 - |\delta|^2$, where $\delta$ is half of the sum of positive roots. 

\section{Infinitesimal Character}\label{laplace infin}

Let $G$ be a Lie group and let $\lieg$, $\lieh$, $\liegc$, and $\liehc$ be the corresponding Lie algebras as above.

For $\mu\in\liehc^*$ define $\chi_\mu(Z)=\mu(\gamma(Z))$ for $Z\in Z(\liegc)$. 
Note that $\gamma(Z)\in U(\liehc)=\mathcal H$, 
and so it makes sense to evaluate $\mu$ on $\gamma(Z)$ -- here we are of course extending $\mu$ to a character of the universal enveloping algebra $\mathcal H$. It follows from the Harish-Chandra isomorphism that all characters of $Z(\liegc)$
are of this form.

Let $\pi$ be an irreducible admissible representation of $G$.
From $\pi$ we can obtain a representation of $Z(\liegc)$. This representation is then a character, called the
infinitesimal character of $\pi$ and denoted by $\chi_\pi$. 
This character is of the form
$\chi_\pi=\chi_\mu$ for some $\mu\in\liehc^*$ as above. This $\mu$ is uniquely determined up to the action of $W$ and it is sometimes also called the infinitesimal character of $\pi$.

The Laplacian eigenvalue of $\pi$ is then obtained by evaluating the infinitesimal character on the Casimir element, i.e., $\chi_\pi(\Omega)$. 



\section{Behavior under Functoriality}\label{laplace proof}

To see what happens to the Laplacian eigenvalue under functoriality, we need to
study what happens to the infinitesimal character.

Let $W_\mathbb R$ be the Weil group. By the local Langlands correspondence, to each $\pi$ corresponds a parameter 
$\varphi: W_\mathbb R\rightarrow$ $^L G$, where $^L G$ is the Langlands dual group. By conjugating $\varphi$ if necessary, we may assume that 
$\varphi(W_\mathbb R)\subset N(^L H^0)$, where $H$ is a maximal torus in $G$. Then $\varphi$ is of the form 
$\varphi(z)=z^\mu\overline z^{\mu^\prime}$ for $z\in\mathbb C^\times\subset W_\mathbb R$, 
where $\mu, \mu^\prime\in X_*(H^\vee)\otimes \mathbb C=X^*(H)\otimes \mathbb C$ are such that $\mu-\mu^\prime\in X^*(H)$. Here $\varphi(z)=z^\mu\overline z^{\mu^\prime}$ stands for 
$\varphi(e^s)=\exp(s\mu+\overline s\mu^\prime)\in$ $^L H^0$ for $s\in\mathbb C$.

By definition we have $X_*(H^\vee)=X^*(H)$, and so we can view $\mu$ as an element of $X^*(H)\otimes \mathbb C$ and attach to it a character $\chi_\mu$ of
$Z(\liegc)$ via the Harish-Chandra homomorphism. If $\varphi$ is the Langlands parameter of $\pi$, then $\chi_\mu$ is the infinitesimal character of $\pi$.
(See for example \cite{av} for details.) 

\

In the rest of this section, let $G=\mathbf G(\mathbb R)$ be a quasisplit classical group. 
Let $\Pi$ be a self-dual irreducible admissible representation of $GL_N(\mathbb R)$ which is a functorial lift of an 
irreducible admissible representation $\pi$ of $G$. The Laplacian eigenvalue of $\Pi$ is determined from the infinitesimal character of the restriction of $\Pi$ to $SL_N(\mathbb R)$, and so we will be dealing with $SL$ and $\liesl$.

Let $\varphi: W_\mathbb R\rightarrow$ $^L G$ be the Langlands parameter of $\pi$. The Langlands parameter of $\Pi$ is then given by composition with the
inclusion $^L G\hookrightarrow$ $^L SL_N$.

\begin{proposition}\label{functoriality infin}
Let $H$ be a maximal torus of $G$ and $T$ a maximal torus of $SL_N(\mathbb R)$ such that $H^\vee\hookrightarrow T^\vee$. 
This induces an inclusion $\iota: X^*(H)\otimes \mathbb C=X_*(H^\vee)\otimes \mathbb C\hookrightarrow X_*(T^\vee)\otimes \mathbb C=X^*(T)\otimes \mathbb C$.

Then the infinitesimal character of $\Pi$ is $\iota(\mu)$, where $\mu$ is the infinitesimal character of $\pi$.
\end{proposition}

\begin{proof}
Immediately follows from the discussion at the beginning of this section.
\end{proof}

\begin{theorem}\label{main laplace}
Let $\lambda$ and $\Lambda$ be the Laplacian eigenvalues of $\pi$ and $\Pi$, respectively.
Then there are constants $c_G>0$ and $d_G$ such that $\Lambda=c_G\lambda+d_G$.

We have $c_{Sp(2n)}=\frac{2n+2}{2n+1}$, $c_{SO(2n)}=c_{SO^\ast(2n)}=\frac{n-1}{n}$, 
$c_{SO(2n+1)}=\frac{2n-1}{2n}$.
\end{theorem}

\begin{proof}
To prove the theorem, we evaluate the infinitesimal characters on Casimir elements using
Proposition \ref{functoriality infin}.

Denote by $E_{ij}$ the matrix with 1 at the position $(i, j)$ and 0 elsewhere and denote $E_{ii}$ also by
$E_i$. Let $E_i^\ast\in\lieh^\ast$ be the corresponding dual element. 
For $A=\sum a_iE_i\in\liehc$, we define $A^\ast=\sum a_iE_i^\ast\in\liehc^\ast$.

\

For $\lieg=\liesl_N(\mathbb R)$, let $\liea$ be the subalgebra of diagonal matrices
with trace zero. Extend the set $E_i-E_{n+i}$ to a basis of $\lieac$ (where $N=2n+\varepsilon$ with $\varepsilon= 0, 1$).
We have $B(E_i-E_{n+i}, E_j-E_{n+j})=2N\Tr((E_i-E_{n+i})(E_j-E_{n+j}))=2N\delta_{ij}$ (where $\delta_{ij}$ is Kronecker delta). 
Hence we can extend $A_i=\frac 1{\sqrt{2N}}(E_i-E_{n+i})$ to an orthonormal basis of
$\lieac$ with respect to the Killing form $B$.

\

Take $G=Sp_{2n}(\mathbb R)$. Take $\liehc$ to be the subalgebra of diagonal matrices in $\lieg$;
we can choose $b_i=E_i-E_{n+i}$ as a basis of $\liehc$.
We have $B(b_i, b_j)=(2n+2)\Tr(b_ib_j)=2(2n+2)\delta_{ij}$. 
Thus $H_i=\frac 1{\sqrt{2(2n+2)}}b_i$ form an orthonormal basis of $\liehc$ with respect to the Killing form $B$.

Let $\pi$ be a representation of $G$ with infinitesimal character $\mu=(\mu_1, \dots, \mu_n)\in\liehc^\ast$ with respect to the basis $b_i^\ast=E_i^\ast-E_{n+i}^\ast$, i.e., 
$\mu=\sum \mu_i b_i^\ast$. We have $\gamma(\Omega)=\sum H_i^2 - |\delta_{Sp}|^2$
and $b_i^\ast(H_j)=\frac 2{\sqrt{2(2n+2)}}\delta_{ij}$, and so
we see that $$\mu(\gamma(\Omega))=\sum\mu_i^2b_i^\ast(H_i)^2-|\delta_{Sp}|^2=
\frac 2{2n+2}\sum\mu_i^2-|\delta_{Sp}|^2.$$


Let $\Pi$ be the restriction of the functorial lift of $\pi$ to $SL_{2n+1}(\mathbb R)$.
The infinitesimal character of $\Pi$ is then $\iota(\mu)=(\mu_1, \dots, \mu_n,-\mu_1, \dots,-\mu_n)\in\lieac^\ast$. 
Note that $\iota(\mu)=\sum\mu_i (E_i-E_{n+i})^\ast$.
Thus the Laplacian eigenvalue
$$\iota(\mu)(\gamma(\Omega))=\sum\mu_i^2(E_i-E_{n+i})^\ast(A_i)^2-|\delta_{SL}|^2=
\frac 2{2n+1}\sum\mu_i^2-|\delta_{SL}|^2.$$ We see that $c_{Sp(2n)}=\frac{2n+2}{2n+1}.$

\

Take $G=SO_{2n}(\mathbb R)$. This is entirely analogous to the case of $Sp_{2n}$,
the only differences are that $B(b_i, b_j)=(2n-2)\Tr(b_ib_j)$ and that $\pi$ lifts to $SL_{2n}(\mathbb R)$, which leads to 
$c_{SO(2n)}=\frac{2n-2}{2n}=\frac{n-1}{n}$.

\ 

The case of $G=SO_{2n+1}(\mathbb R)$ is also analogous, we get 
$c_{SO(2n+1)}=\frac{2n-1}{2n}$.

\

Let us finally discuss the case when $G=SO^\ast_{2n}(\mathbb R)$ in more detail.
Take $G$ and $\lieg$ to be defined with respect to the matrix $M=(m_{ij})$ with 
$m_{ij}=1$ for $i+j=2n$, $1\leq i\leq n-1$, or $n\leq i=j\leq n+1$, and $m_{ij}=0$ otherwise.
Choose the Cartan algebra $\lieh$ to be the span of $b_j=E_j-E_{2n-j}$ for $1\leq j\leq n-1$, 
and $b_n=E_{n, n+1}-E_{n+1, n}$. Note that in this case the Killing form is not 
positive definite on $\lieh$, as we have $B(b_n, b_n)=-2$.
Setting $H_j=\frac 1{\sqrt{2(2n-2)}}b_j$ we get that $$\gamma(\Omega)=\sum_{j=1}^{n-1} H_j^2 - H_n^2 -|\delta_{SO^\ast}|^2.$$

Let $\pi$ be a representation of $G$ with infinitesimal character $\mu=(\mu_1, \dots, \mu_n)\in\liehc^\ast$ with respect to the basis $b_j^\ast$. 
Then $$\mu(\gamma(\Omega))=\sum_{j\leq n-1}\mu_j^2b_j^\ast(H_j)^2-\mu_n^2b_n^\ast(H_n)^2 -|\delta_{SO^\ast}|^2=
\frac 2{2n-2}(\sum_{j\leq n-1}\mu_j^2-\mu_n^2)-|\delta_{SO^\ast}|^2.$$

Let $\Pi$ be the restriction of the functorial lift of $\pi$ to $SL_{2n}(\mathbb R)$ and let $\varphi: W_\mathbb R\rightarrow$ $^L G$ be the Langlands parameter of $\pi$.
Viewing $\mu$ as a cocharacter of $H^\vee$ we get that $\mu(\mathbb C^\times)$ is not in the torus $^LT^0$ -- we have to conjugate the part corresponding to $H_n$, originally of the form $\left(\begin{array}{cc} \cos\mu_n & \sin\mu_n \\ -\sin\mu_n & \cos\mu_n \end{array}\right)$, which we conjugate to $\left(\begin{array}{cc} e^{i\mu_n} & 0 \\ 0 & e^{-i\mu_n} \end{array}\right)$ (here $i=\sqrt{-1}$).
This gives us that the infinitesimal character of $\Pi$ is then $\iota(\mu)=(\mu_1, \dots, i\mu_n,-i\mu_n, \dots,-\mu_1)\in\lieac^\ast$. 
Hence
$$\iota(\mu)(\gamma(\Omega))=
\frac 2{2n}(\sum\mu_j^2-\mu_n^2)-|\delta_{SL}|^2.$$ We conclude that
$c_{SO^\ast(2n)}=\frac{2n-2}{2n}=\frac{n-1}{n}$.

\

This concludes the proof of Theorem \ref{main laplace}. Note that one can also easily compute the constants $d_G$ by calculating the various values $|\delta|^2$.
\end{proof}

%
%
%

\chapter{Depth}

In this chapter we shall study representations of $p$-adic groups, in particular properties of their depth. A powerful tool for that is 
provided by Bruhat-Tits theory. We first briefly review some aspects of it in Section \ref{depth-backgr} and 
then, in Section \ref{depth-main}, we prove the results we need for counting self-dual representation. 
In Section \ref{depth-descent} we use the local descent to show a partial result towards depth preservation for split classical groups. However, depth need not be preserved in the case of wildly ramified groups, as we note in Section \ref{depth non}.
The last Section, \ref{depth restr}, is devoted to the study of the restriction of certain representations of $SO_{2n}$ to $SO_{2n-1}$, which is one of the steps needed in trying to extend the results of Section \ref{depth-descent} to a proof of full depth preservation, not just one inequality.

Throughout this chapter, all groups and representations are $p$-adic, unless stated otherwise.

\section{Background}\label{depth-backgr}

Throughout this chapter, let $F$ be a local non-archimedean field (of characteristic zero) and $G$ a connected reductive group over $F$.
We shall study irreducible admissible representations of $G(F)$, their depth and behaviour under functoriality, so let's first review 
some background on Bruhat-Tits theory, Moy-Prasad filtrations, and depth. The basic references that we shall follow are \cite{tits}, \cite{yuottawa},
and \cite{mp2}.


Let $S$ be a maximal $F$-split torus of $G$, $N(S)$ the normalizer of $S$ in $G$, and $Z(S)$ the centralizer of $S$ in $G$. 
Let $C$ be the maximal $F$-split torus contained in the center of $G$, we have $C\subset S$.
Let $(\Phi, R, \Phi^\vee, R^\vee)$ be the relative root system, let $W$ be its Weyl group and
$\widetilde{W}$ the affine Weyl group, an extension of $W$ by a free abelian group of rank $\dim S$. 
The (restricted) apartment $A=A(G, S, F)$ is defined as an affine space under $(X_*(S)\otimes \mathbb R)/(X_*(C)\otimes \mathbb R)$, it is equipped by an action of $\widetilde W$
and by a corresponding action of $N(S)$. We denote both of these action by $\nu$. See \cite {tits}, 1.2 for details. 
Note that \cite {tits} works with the extended apartment, but the reduced one is somewhat more convenient for our purposes.

We also have the set of affine roots $\Phi_{\mathrm{af}}$, being the set of certain affine functions $\alpha=a+\gamma$ for $a\in\Phi$ and $\gamma\in\mathbb R$. 
See \cite {tits}, 1.6.

For $a\in\Phi$ we have the root subgroup $U_a$; for an affine function $\alpha=a+r$, $r\in\mathbb R$ one defines
a subgroup $X_\alpha$ of $U_a$ as the subset of elements $u$ satisfying $u=1$ or $\alpha(a, u)\geq \alpha$, where $\alpha(a, u)$ is a certain affine
function, defined in \cite {tits}, 1.4.

For an affine function $\alpha=a+r$ with $a\in\Phi$, we denote $A_\alpha=\alpha^{-1}([0, \infty))$, 
its boundary is $\partial A_\alpha=\alpha^{-1}(0)$, and $r_\alpha$ is the affine reflection
whose vector part is the reflection $r_a$ and whose fixed hyperplane is $\partial A_\alpha$.
When $\alpha\in\Phi_{\mathrm{af}}$, we call $A_\alpha$ the half-apartments and
$\partial A_\alpha$ the walls. The chambers are the connected components of
the complement in $A$ of the walls (cf. \cite{tits}, 1.7).
A chamber is a polysimplex (i.e., a product of simplices); by a facet of $A$ we mean just 
a facet of some chamber (viewed as a polysimplex).

The (restricted) building $\mathcal B=\mathcal B(G, F)$ is defined as the unique set equipped with a left $G$-action such that 
$\mathcal B=\bigcup_{g\in G} gA$, the group $N(S)$ stabilizes $A$ and acts on it by $\nu$, and for every affine root $\alpha$, the group
$X_\alpha$ fixes the half-apartment $A_\alpha$ (cf. \cite{tits}, 2.1).

Fix a chamber $\mathcal C$. Then each element of $A$ is $N(S)$-conjugate to an element of the closure $\bar{\mathcal C}$. Hence every element of $\mathcal B$ is $G$-conjugate to an element of the closure $\bar{\mathcal C}$.

For a facet $\mathcal F$, one can define a corresponding parahoric subgroup $G_\mathcal F$. The definition is somewhat involved (\cite{yuottawa}, 2.2.4), 
but roughly speaking, the parahoric subgroup is essentially the stabilizer of $\mathcal F$ in $G$. 
One can also define the parahoric subgroup $G_x$ for $x\in\mathcal B$; 
if $x$ is a ``generic" point of a facet $\mathcal F$, then $G_\mathcal F=G_x$.

\

We have the Moy-Prasad subgroups which are used to define the depth of admissible representations \cite{mp2}. 
First, denote by $Z_n$ the group of all $z\in Z(S)$ such that $\omega(\chi(z)-1)\geq n$ for all characters $\chi$ of $Z(S)$
($\omega$ is the additive valuation on $F$).
For $x\in \mathcal B$ and $r\geq 0$, we define
$G_{x, r}$ to be the intersection of $G_x$ with the subgroup generated by $Z_n$ for $n\geq r$ and by 
$X_\alpha$ for all $\alpha$ satisfying $\alpha(x)\geq r$. 
Denote by $G_{x, r+}$ the union of all $G_{x, s}$ with $s>r$.

Let $\pi$ be an irreducible admissible representation of $G$. 
The depth of $\pi$ is defined as 
$\depth\pi=\inf\{r\geq 0| \exists x\in\mathcal B$ such that $\pi^{G_{x, r+}}\neq 0\}$.

One also defines the depth of a Weil-Deligne 
representation $\varphi: W_F^\prime\rightarrow$ $^LG$ as $\inf\{r\ |\ \varphi(I_s)=1$ for all $s>r \}$, where $I_s$, $s\in\mathbb Q$, is the filtration 
of the inertia subgroup $I$ of $W_F\subset W_F^\prime$.

We have the following fundamental properties of depth:
\begin{itemize}
\item $\depth\pi$ is a rational number and the infimum is achieved for some $x\in\mathcal B$.
\item The depth is preserved by parabolic induction and the Jacquet functor. Namely, let
$M$ be a Levi of $G$ and $\sigma$ an irreducible admissible representation of $M$.
If $\pi$ is an irreducible subquotient of the normalized induction $\mathrm{Ind}_M^G\sigma$, then $\depth\sigma=\depth\pi$
(and similarly for the Jacquet functor).
\item Depth preservation: Assume that $\mathbf G$ is unramified or tamely ramified. If $\varphi:W_F^\prime\rightarrow$ $^LG$ is the Langlands parameter of a representation $\pi$, then it is conjectured that $\depth\pi=\depth\phi$. J.-K. Yu \cite{yuottawa} has proved this 
for $GL_N$ and for tamely ramified tori.
\end{itemize}

As we note in Section \ref{depth non}, depth need not be preserved in the case of wildly ramified groups. This has also been observed by Gross-Reeder \cite{gross-reeder};
Reeder-Yu \cite{reeder-yu} recently provided more examples of this phenomenon.

\section{Bounds on Depth}\label{depth-main}

We will use the following result on uniform admissibility, due originally to Howe when $\mathbf G=GL_N$, and 
then to Bernstein for general $\mathbf G$ (with a simpler proof) -- for part b) see \cite{bernstein}, Theorem 1; part a) then follows as in \cite{bz}, Corollary 4.15.

\begin{theorem}\label{Bernstein}(\cite{bernstein}, \cite{bz})
Let $\mathbf G$ be a reductive group over a non-archimedean local field $F$. Let $K$ be an open compact subgroup of $G=\mathbf G(F)$. Then:

a) There are finitely many supercuspidal representations $\sigma$ of $G$ such that $\sigma^K\neq 0$.

b) There is an $n\in\mathbb N$ (depending only on $G$ and $K$) such that, for all 
irreducible admissible representations $\pi$ of $G$, we have $\dim \pi^K\leq n$.
\end{theorem}

Both parts of Theorem \ref{Bernstein} will be useful for us in the thesis. 
We shall shortly use the first part to prove part b) of Proposition \ref{bound-depth}.
The second part of the theorem will play an important role in the arguments of Chapter 5, where it will enable us to uniformly bound the dimensions of the spaces of $K$-fixed vectors of the self-dual representations which we are counting.

\begin{proposition}\label{bound-depth}
Let $\mathbf G$ be a reductive group over a non-archimedean local field $F$. 

a) For each $r\geq 0$ there is an open compact subgroup $K$ of $G=\mathbf G(F)$ 
such that if $\depth\pi\leq r$, then $\pi^K\neq 0$.
When $\mathbf G$ is split and almost simple, we can take $K$ to be the principal congruence subgroup $K_m$
with $m=\left\lceil r \right\rceil +1$.

b) There are finitely many supercuspidal representations $\sigma$ of $G$ such that  $\depth\sigma$  $\leq r$.
\end{proposition}

\begin{proof}
a) In the split case this is due to Radhika Ganapathy \cite{ganapathy}, Lemma 8.2. The proof in the general case is similar to the split one:

Let $\sigma$ be an irreducible admissible representation of $G$ with $\depth\sigma\leq r$. 

Let $\mathcal C$ be a chamber of $\mathcal B$.
For $a\in\Phi$,  $a(x)$ is a continuous function on the compact set $\bar{\mathcal C}$, and so it has a minimum, 
which we denote $m(a)$. Define then $G_r$ to be the subgroup of $G$ generated by all the 
$Z_n$ with $n>r$ and by all $U_\alpha$ with $\alpha=a+\gamma$
satisfying $\gamma>r-m(a)$. We shall show that $\sigma$ has a vector fixed by $K=G_r$.

For each $x\in\mathcal B$ there is $y\in\bar{\mathcal C}$ and $g\in G$ such that $x=g\cdot y$. Then $G_{x, r}=gG_{y, r}g^{-1}$. Thus if $\depth\sigma\leq r$, there is $x\in \mathcal B$ such that $\sigma^{G_{x, r+}}\neq 0$, and therefore also
$\sigma^{G_{y, r+}}\neq 0$.

Take now $\alpha=a+\gamma$ such that $\gamma> r-m(a)$. We have $\alpha(y)=a(y)+\gamma> m(a)+(r-m(a))=r$, and so we see that the subgroup
$U_\alpha$ is contained in $G_{y, r+}$. Since this is true for any such $\alpha$ and also for all $Z_n$ with $n>r$, we conclude that 
$G_r\subset G_{y, r+}$. Hence $\sigma^{G_{r}}\neq 0$.

b) Take the subgroup $K=G_r$ from part a). By Theorem \ref{Bernstein} a) there are only finitely many supercuspidals having a vector fixed by $K$,
and so there are also only finitely many supercuspidals of depth $\leq r$.
\end{proof}

Let $\mathbf G$ be a quasisplit classical group over a non-archimedean local field $F$, i.e., $G$ is split $SO_{2n+1}$, $Sp_{2n}$ or $SO_{2n}$, or quasisplit (non-split) even orthogonal corresponding to a quadratic extension $E/F$. By Arthur's work we know the local Langlands correspondence for $G$ (conditionally on the stabilization of the twisted trace formula) and a slightly weaker form of it when $\mathbf G$ is even orthogonal. We know that irreducible admissible representations of $G$ lift to self-dual (irreducible admissible) representations of $GL_N(F)$ for a suitable $N$.

Let $^L G$ be the Langlands dual group of $G$ and $W_F^\prime$ the Weil-Deligne group. For a representation $\pi$ of $G$ we denote $\phi_\pi: W_F^\prime \rightarrow$ $^L G$ the corresponding Langlands parameter and $\Pi$ the lift to $GL_N(F)$ with parameter $\iota\circ\phi_\pi: W_F^\prime \rightarrow$ $^L G\rightarrow$ $^L GL_N$ (all this being up to outer conjugacy by $O_{2n}$ in the even orthogonal case).
Note that $\depth\phi_\pi=\depth\iota\circ\phi_\pi$.

\begin{proposition}\label{rs}
Fix $\mathbf G$ and $F$ as above. Then for each $r\geq 0$ there exists $s\geq 0$ such that:

a) For all irreducible admissible representations $\pi$ of $G$  if $\depth \phi_\pi=\depth \Pi\leq r$, then $\depth\pi\leq s$.

b)  For all irreducible admissible representations $\pi$ of $G$ if $\depth\pi\leq r$, then $\depth \phi_\pi=\depth \Pi\leq s$.
\end{proposition}

\begin{proof}
a) Assume first that $G$ is not an even orthogonal group, so that we know the full local Langlands correspondence for $G$.

1) We first prove the statement for supercuspidal representations. Let $\sigma$ be a supercuspidal representation of $G$. Let $\phi_\sigma$ be its Langlands parameter and $\sigma^\prime$ the lift to $GL_N(F)$. Assume that $\depth \phi_\sigma=\depth \sigma^\prime\leq r$. 

The representation $\sigma^\prime$ is not necessarily supercuspidal, but there will be a Levi $M$ of $GL_N(F)$ and a supercuspidal representation $\tau$ of $M$ so that $\sigma^\prime$ is a subquotient of the normalized parabolic induction $\mathrm{Ind}_M^{GL_N(F)} \tau$. By the general properties of depth, we have 
$\depth \tau=\depth \sigma^\prime\leq r$.

Now, there are only finitely many standard Levis $M$ (including the one $M=GL_N(F)$).
By Proposition \ref{bound-depth}, for a fixed $M$ there are only finitely many supercuspidal representations $\tau$ of $M$ with $\depth\tau\leq r$. 
And for fixed $M$ and $\tau$, the parabolic induction $\mathrm{Ind}_M^{GL_N(F)} \tau$ has only finitely many possible constituents $\sigma^\prime$ by \cite{cass}, Corollary 6.3.7.
Let $\Sigma^\prime$ be the finite set of all triples $(M, \tau, \sigma^\prime)$ as above, and let $\Sigma$ be the finite set of 
supercuspidal representations of $G(F)$ that lift to some $\sigma^\prime$ belonging to a triple $(M, \tau, \sigma^\prime)\in\Sigma^\prime$. Define $s$ as the maximum of the depths of all representations $\sigma_0\in\Sigma$.

Thus the representation $\sigma$ with which we started lies in $\Sigma$, and so $\depth\sigma\leq s$.

2) Now let $\sigma$ be a supercuspidal representation of some (standard) Levi subgroup $M$ of $G(F)$.

Then $M=\mathbf G_{k_0}(F)\times GL_{k_1}(F)\times\cdots\times GL_{k_l}(F)$, where $\mathbf G_{k_0}$ is a group of the same type as $\mathbf G$ and smaller rank $k_0$. Correspondingly we have $\sigma=\sigma_0\times\sigma_1\times\cdots\times\sigma_l$, where $\sigma_i$ are supercuspidal representations of $G_{k_0}$ or $GL_{k_i}(F)$. Let $\phi_i$ be the Langlands parameter of $\sigma_i$, $\sigma_0^\prime$ the lift of $\sigma_0$ to suitable $GL_{k_0^\prime}(F)$, and $M^\prime=GL_{k_0^\prime}(F)\times GL_{k_1}(F)\times\cdots\times GL_{k_l}(F)$. Then $\phi=\phi_0\oplus\phi_1\oplus\cdots\oplus\phi_l$ is the Langlands parameter of $\sigma$.

We have $\depth\sigma=\max(\depth\sigma_i)$ and $\depth\phi=\max(\depth\phi_i)$. By the depth preservation theorem for $GL$, we know that $\depth\sigma_i=\depth\phi_i$ for $i>0$. By the first part of the proof we now get the statement of the proposition in this case.

Since there are finitely many standard Levi subgroups, we can define $s$ to be the maximum of numbers $s_M$ obtained above for individual Levis $M$.

3) Finally, let $\pi$ be any representation of $G(F)$ with $\depth \phi_\pi=\depth \Pi\leq r$.
There is a standard Levi $M$ of $G$ and a supercuspidal representation $\sigma$ of $M$ such that $\pi$ is a subquotient of the normalized induction $\mathrm{Ind}_M^G \sigma$. By the properties of depth we have $\depth \sigma=\depth\pi$. Let $\phi_\sigma$ be the Langlands parameter of $\sigma$. We can also see that
$\depth \phi_\sigma=\depth\phi_\pi$. By our assumption, this depth $\leq r$, and so $\depth \sigma=\depth\pi\leq s$, as we wanted to show.

4) Now it remains to deal with the even orthogonal case. In this case we know the local Langlands correspondence only up to outer conjugacy by $O_{2n}$.
Given a representation $\pi$ of $G$, we know that its Langlands parameter will be one of two possible parameters $\phi_1$, $\phi_2$. These parameters are conjugate by $O_{2n}$ in $^L G$, and so $\depth\phi_1=\depth\phi_2$. Now we can use the proof above both for $\phi_1$ and $\phi_2$ and define $s$ to be the maximum of the two numbers $s_1, s_2$ that we obtain.

b) The proof is entirely analogous to part a), so we shall omit it.
%
%
%
\end{proof}

For $m\in\mathbb N$ denote by $K_m$ the principal congruence subgroup of $GL_N(F)$ of level $m$, i.e., $K_m$ consists of matrices congruent to the 
identity $I$ modulo $\wp^m$ (where $\wp$ is the prime ideal of $\mathcal O_F$).

\begin{corollary}\label{depth cor}
Let $\mathbf G$ be a quasisplit classical group.
For each $m\in\mathbb N$ there exists a compact open subgroup $K\subset G$ such that:

For each self-dual irreducible admissible representation $\Pi$ of $GL_N(F)$ which descends to a representation $\pi$ of $G$ 
and satisfies $\Pi^{K_m}\neq 0$, one has $\pi^K\neq 0$.
%
\end{corollary}

\begin{proof}
Since $\Pi^{K_m}\neq 0$, we see that $\depth\Pi\leq m$. Thus by Proposition \ref{rs}, there is $s$ (depending only on $m$ and not on $\Pi$)
such that $\depth\pi\leq s$. By Proposition \ref{bound-depth} a), there is $K$ (again depending only on $s$) such that $\pi^K\neq 0$.
\end{proof}

\section{Local Descent}\label{depth-descent}

We can also use the construction of local descent to obtain a sharper estimate on the depths in Proposition \ref{rs} 
in the case of generic representations. 

\begin{theorem}\label{descent-bound}
Let $F$ be a $p$-adic field and $G=SO_{2n+1}(F)$.  
Assume that $p$ is sufficiently large.

a) Let $\Pi$ be an (irreducible admissible) representation of $GL_n(F)$ that descends to a (generic) representation $\pi$ of $G$. Then $\depth\Pi\geq\depth\pi$.

b) Let $\pi$ be an (irreducible admissible) generic representation of $G$ and $\phi_\pi$ its Langlands parameter. Then $\depth\pi\leq\depth\phi_\pi$.
\end{theorem}

Before proving the theorem, 
let us review the local descent for the case $\mathbf G=SO_{2n+1}$, following \cite{js}.

Let $\tau$ be a self-dual supercuspidal representation of $GL_{2n}(F)$  that descends to $G=SO_{2n+1}(F)$ (rather than to an even orthogonal group). Denote $H=SO_{4n}(F)$.
First view $GL_{2n}(F)$ as the Siegel Levi of a parabolic $P$ in $H$ and take the normalized induction $\rho_\tau=\mathrm{Ind}_P^H \tau |\det|^{1/2}$. Let $\pi_\tau$ be the Langlands quotient of $\rho_\tau$.

The group $H=SO_{4n}(F)$ is the group of isometries of a $4n$-dimensional vector space $V$ preserving 
a symmetric bilinear form $b$. Fix maximal isometric subspaces $V^+$ and $V^-$ of $V$ in duality with respect to $b$. Also fix a maximal flag 
$0\subset V_1^+\subset V_2^+\subset\dots\subset V_n^+=V^+$ and a corresponding basis 
$\{e_1, \dots, e_n\}$ such that $V_i^+=\mathrm{Span} (e_1, \dots, e_i)$. Let $\{e_{-1}, \dots, e_{-n}\}$ be the dual basis of $V^-$ and denote $V_{i}^-=\mathrm{Span} (e_{-1}, \dots, e_{-i})$.
Finally, let $W=(V_{n-1}^+ + V_{n-1}^-)^\perp$.

Let $P_{n-1}=M_{n-1}N_{n-1}$ be the parabolic subgroup of $H$ stabilizing the flag 
$0\subset V_1^+\subset V_2^+\subset\dots\subset V_{n-1}^+$. Then $M_{n-1}$ is isomorphic
to $SO(W)\times GL_1(F)^{n-1}$, note that $SO(W)\simeq SO_{2n+2}(F)$.

Fix $\alpha\in F^\times$ and denote $y=e_{2n}-\alpha^2 e_{-2n}$. Let $L_{n-1}$ be 
the stabilizer of $y$ in $SO(W)$, we have 
$L_{n-1}=SO(\{y\}^\perp\cap W)\simeq SO_{2n+1}(F)(\subset SO_{2n+2}(F))$
(for more detail on this embedding, see Section \ref{depth restr}).

Let $\psi$ be an unramified additive character of $F$ and let $\psi_{n-1}$ be the corresponding 
character of $N_{n-1}$. For a smooth representation $\pi$ of $H$
we shall denote by $J_{\psi_{n-1}}(\pi)$ the twisted Jacquet module of $\pi$ with respect to $N_{n-1}(F)$ and its character $\psi_{n-1}$, i.e., 
$J_{\psi_{n-1}}(\pi)=V_\pi / \mathrm{Span}\{\pi(a)v-\psi_{n-1}(a)v | a\in N_{n-1}, v\in V_\pi\}$. 
Note that, as a representation of the Levi $M_{n-1}$, this is the same as the (usual, non-twisted) Jacquet module 
$J(\pi\otimes \psi_{n-1}^{-1})=V_\pi / \mathrm{Span}\{\pi(a)\psi_{n-1}(a)^{-1}v-v | a\in N_{n-1}, v\in V_\pi\}$.

Finally, let $\sigma=\sigma_{\psi, n-1}(\tau)$ be the restriction of the twisted Jacquet module to $L_{n-1}$, $\sigma=\mathrm{Res}_{L_{n-1}}J_{\psi_{n-1}}(\pi_{\tau})$.

Then $\sigma$ is the descent of $\tau$ and it is an irreducible generic supercuspidal representation of $L_{n-1}\simeq SO_{2n+1}(F)$. 

Let us now prove Theorem \ref{descent-bound}; note that the local descent construction applies in the cases of other classical groups as well,
only generally it may happen that $\sigma$ is not irreducible, and the descent is in fact an irreducible subquotient of $\sigma$. 
The result and its proof are similar, so we deal only with the case of $G=SO_{2n+1}$ for simplicity of notation.

\begin{proof}
a) It suffices to prove the theorem only for supercuspidal representations, in general one proceeds as in the proof of Proposition \ref{rs}, part 3).

So let us take a supercuspidal self-dual representation $\tau$ of $GL_{2n}(F)$ that descends to $G=SO_{2n+1}(F)$.
The character $|\det|^{1/2}$ is trivial on $GL_{2n}(\mathcal O)$, and so it has depth zero, and hence $\depth\tau=\depth\tau |\det|^{1/2}$.
Since parabolic induction preserves depth and $\pi_\tau$ is an irreducible subquotient of $\rho_\tau=\mathrm{Ind}_P^H \tau |\det|^{1/2}$,
we see that $\depth\tau=\depth\pi_\tau$.

The character $\psi_{n-1}^{-1}$ is unramified, and so 
$\depth\pi_\tau=\depth(\pi_\tau\otimes \psi_{n-1}^{-1})$. Since depth is preserved under
taking a Jacquet quotient, we see that 
$\depth\tau=\depth(\pi_\tau\otimes \psi_{n-1}^{-1})=\depth (J_{\psi_{n-1}}(\pi_{\tau}))$,
where we consider $J=J_{\psi_{n-1}}(\pi_{\tau})$ as a representation of $M_{n-1}$.

We are assuming that the residual characteristic $p$ is large enough, and so the 
Moy-Prasad filtration on $L_{n-1}\simeq SO_{2n+1}(F)$ is obtained as the intersection of 
the filtration on $M_{n-1}\simeq SO_{2n+2}(F)$ with $L_{n-1}$ (see \cite{yuottawa}, 2.3.10.3).
Clearly, if $J$ has a vector fixed by some subgroup $(M_{n-1})_{x,r}$, then it has also a vector
fixed by $(M_{n-1})_{x,r}\cap L_{n-1}$. Thus $\depth \mathrm{Res}_{L_{n-1}}J\leq \depth J=\depth\tau$,
as we wanted to prove.

b) Easily follows from a).
\end{proof}

Let us note that the reason for why we get only one inequality towards depth preservation is the
last step of taking the restriction, which
may generally decrease the depth. It seems plausible that studying the local descent more closely can prove 
the full depth preservation for generic representations -- we outline a partial progress in this direction in Section \ref{depth restr}.

\section{Depth Preservation and Wildly Ramified Groups}\label{depth non}

In this section we give an example when depth is not preserved under 
Langlands reciprocity and functoriality. 
This occurs in some of the examples of residual characteristic $p=2$ discussed in Section 7.2 of \cite{cps}. 
The fact that depth need not be preserved when residual characteristic is small has already been observed by Gross-Reeder \cite{gross-reeder};
Reeder-Yu \cite{reeder-yu} recently provided more examples of this phenomenon.

Let $F=\mathbb Q_p$ and $E=\mathbb Q_p(\sqrt {p})$. Let $\mathbf G$ be the non-split, quasisplit orthogonal group $SO_{2n}^*$ attached to the quadratic extension
$E/F$, let $G=\mathbf G(F)$.

Let $\eta=\eta_{E/F}: F^\times \rightarrow \mathbb C$ be the quadratic character associated to $E/F$ by local class field theory.
Note that $\eta$ is defined by $\ker \eta=N_{E/F}(E^\times)$.

It is easy (and classical) to observe that $\eta$ is ramified and to compute its depth:

\begin{lemma}\label{eta}
$\eta$ is ramified. If $p>2$, then $\depth\eta=0$. If $p=2$, then $\depth\eta=2$.
\end{lemma}

\begin{proof}
We want to check that $\eta|\mathbb Z_p^\times\neq 1$. Take $a\in\mathbb Z_p^\times$ such that $a$ is not a square modulo $p$. 
We shall show that $a$ is not a norm. Assume false, i.e., we have $a=N(x+\sqrt {p} y)=x^2-py^2$ for some $x, y\in\mathbb Q_p$. Let $p^m$ be the common denominator of $x$ and $y$ and write
$x=p^{-m}b, y=p^{-m}c$ with $b, c\in\mathbb Z_p$ and $m\geq 0$. We obtain $p^{2m}a=b^2-pc^2$. 

If $m>0$, then $p$ divides $b$ and we can write $b=pd$ and we get $p^{2m-1}a=pd^2-c^2$. Thus $c=pe$ and 
$p^{2m-2}a=d^2-pe^2$. Proceeding in this way, we eventually obtain an equation with $m=0$.

Let's thus assume that $m=0$ and $a=b^2-pc^2$. We see that if $a$ is a norm, it has to be a square modulo $p$.

Hence $\eta$ is non-trivial on $\mathbb Z_p^\times$, i.e., $\eta$ is ramified.

The level $r$ subgroup $G_r$ of the Moy-Prasad filtration on $\mathbb Q_p^\times$ is just
$1+p^{\lceil r\rceil}\mathbb Z_p$, where $\lceil r\rceil$ is the smallest integer with $\lceil r\rceil\geq r$. 
We then have $G_{r+}=\bigcup_{s>r}G_s=1+p^{r+1}\mathbb Z_p$ for $r\in\mathbb N$, and
$G_{r+}=G_r=1+p^{\lceil r\rceil}\mathbb Z_p$ for $r\not\in\mathbb N$.

When $p>2$, we can see by Hensel's Lemma that each element of $1+p\mathbb Z_p$ is a norm, and so $\depth\eta=0$.
When $p=2$, every element of $1+8\mathbb Z_2$ is a norm, whereas 5 is not a norm. Hence
$\depth\eta=2$.
\end{proof}

Let now $p=2$.
Take any irreducible, admissible and generic representation $\pi$ of $G$ of depth zero and let $\Pi$ be its lift to $GL_{2n}(\mathbb Q_2)$.
The central character of $\Pi$ is $\eta$, and so $\depth\Pi\geq\depth\eta=2$.
In fact, in the explicit cases discussed in Section 7.2 of \cite{cps}, we see that  $\depth\Pi=2$ (take for example $\mu_1=\dots=\mu_n=1$, then $\Pi$ is just the induction
of the depth one character $(1, \dots, 1, \eta, 1, \dots, 1)$ with $\eta$ at the $n$-th coordinate,
of the torus $GL_1^{2n}(\mathbb Q_2)$).

Note that $\Pi$ is not unramified, even when $\pi$ was. This is no contradiction, 
as the fact that unramified representations lift to unramified ones (via Satake isomorphism) holds only for quasisplit groups which split over an unramified extension, i.e., an unramified group, which is not our case. To sum it up, we have:

\begin{corollary}\label{depth non-pres}
Let $\pi$ be an irreducible, admissible, generic, unramified representation of the ramified group $G=SO_{2n}^*(\mathbb Q_2)$ and $\Pi$ its functorial lift to $GL_{2n}(\mathbb Q_2)$.
Then $\Pi$ is not unramified and $\depth\pi=0<2\leq\depth\Pi$.
\end{corollary}

\section{Restriction and Depth}\label{depth restr}

In the local descent construction in the previous section, we obtained a representation $J_{\psi_{n-1,\alpha}}(\pi_{\tau})$ of 
$M_{n-1}\simeq SO_{2n+2}(F)\times (GL_1(F))^{n-1}$ and restricted it to $SO_{2n+1}(F)\subset SO_{2n+2}(F)$. The depth of the representation may have 
decreased in this step. The representation of $M_{n-1}$ is a product 
$\sigma\otimes \chi_1\otimes\cdots\otimes \chi_{n-1}$, where $\sigma$ is a representation of $SO_{2n+2}(F)$ and $\chi_i$ are characters of $GL_1(F)$.

Hence if we wanted to show that depth does not decrease under this restriction, 
first we would need to know what are the characters $\chi_i$, e.g., whether they are maybe always unramified in the descent to $SO_{2n+1}$. The author doesn't know this at present.

Second, we need to understand the restriction of $\sigma$ from $SO_{2n+2}(F)$ to $SO_{2n+1}(F)$. In this section we show that the depth
indeed does not change under this restriction. Let us discuss now the situation slightly more generally:

\

In the rest of this section let $F$ be a $p$-adic field, $p\neq 2$. We will consider the (split) orthogonal group of a quadratic space $V$. 
Let $V$ be an $N$-dimensional $F$-vector space. If $N=2n$, choose a basis $x_1, \dots, x_{n}, x_{-1}, \dots, x_{-n}$ of $V$ and equip 
$V$ with the quadratic form $Q(\sum_{i=-n}^{n} a_ix_i)=\sum_{i=1}^n a_i a_{-i}$.
If $N=2n+1$, choose a basis $x_0, x_1, \dots, x_{n}, x_{-1}, \dots, x_{-n}$ of $V$ and equip $V$ with the quadratic form $Q(\sum_{i=-n}^{n} a_ix_i)$ $=a_0^2+\sum_{i=1}^n a_i a_{-i}$.
In both cases $SO(V)=\{g\in SL(V) | gv=v$ for all $v\in V\}=SO_N(F)$.

We'll be considering the embeddings $SO_{2n-1}(F)\subset SO_{2n}(F)\subset SO_{2n+1}(F)$ defined as follows: Take a vector $v\in V$ such that $Q(v)\neq 0$ and take $W=\{v\}^\perp$. Then $SO(W)=SO_{N-1}(F)$ is the stabilizer of $v$ in $SO(V)$.

\

From now on assume that $N=2n$ is even and take $v=x_n-a^2x_{-n}$ for some $a\in F^\times$ (this is the situation that
occurs in the local descent from $GL_{2n-2}$ to $SO_{2n-1}$).

\begin{proposition}
Let $SO(W)=SO_{2n-1}(F)\subset SO(V)=SO_{2n}(F)$ be an embedding as above. Then there is a quadratic space $U\supset V$ with a quadratic form extending 
the one on $V$ (hence we will also denote it by $Q$) such that $SO(W)\times GL_1(F)$ is a Levi subgroup of $SO(U)=SO_{2n+1}(F)$ 
(where we take $SO(W)\subset SO(U)$ by the natural inclusion). Moreover, $GL_1(F)\cap SO(V)=\{1\}$.
\end{proposition}

\begin{proof}
Take $U$ to be the vector space with basis $x_0, x_1, \dots, x_{n}, x_{-1}, \dots, x_{-n}$ 
(where $x_1, \dots, x_{n}, x_{-1}, \dots, x_{-n}$ is a basis of $V$ and $x_0\not\in V$),
extend the quadratic form to $U$ by $Q(\sum_{i=0}^{2n} a_ix_i)=a_0^2+\sum_{i=1}^n a_i a_{n+i}$.
Then $V$ is the orthogonal complement of $x_0$ in $U$, which gives $SO(V)=SO_{2n}(F)\subset SO(U)=SO_{2n+1}(F)$.

To show that $SO(W)\times GL_1(F)$ is a Levi in $SO(U)=SO_{2n+1}(F)$, we want to find a corresponding isotropic flag $0\subset U_0\subset U$
fixed by a parabolic $P$.
Let $U_0$ be the 1-dimensional subspace spanned by $u_0=ax_0+x_n-a^2x_{-n}$. We have $Q(u_0)=a^2-a^2=0$, and so $U_0$ is isotropic
and $SO(U_0)=GL_1(F)$. By definition, $SO(W)$ fixes $x_0$ and $x_n-a^2x_{-n}$, and so $SO(W)\times GL_1(F)\subset P$ is the Levi.
\end{proof}

\begin{proposition}\label{restr compact}
Let $\pi$ be an irreducible admissible representation of $G=SO_{2n}(F)$ $=SO(V)$ such that its restriction 
$\pi^\prime$ to $H=SO(W)=SO_{2n-1}(F)\subset G$ is also irreducible (and automatically admissible). Assume moreover that $\pi^\prime$ is supercuspidal and generic. 

a) Let $\chi$ be a unitary character of $F^\times$ such that $\chi^2\neq \mathbf 1$. Then $\mathrm{Ind}_{H\times GL_1(F)}^{G\times GL_1(F)} \pi^\prime\otimes \chi\simeq \pi\otimes \chi$.

b) Let $K$ be a compact open subgroup of $G$. Then $\pi^K\neq \{0\}$ if and only if $\pi^{\prime K\cap H}\neq \{0\}$.
\end{proposition}

\begin{proof}
a) Take $U$ as in the previous proposition so that $H\times GL_1(F)$ is a Levi subgroup of $SO(U)=SO_{2n+1}(F)$ 
, where we take $H=SO(W)\subset G=SO(V)\subset SO(U)$ by the natural inclusions. 
Consider the parabolic induction $\mathrm{Ind}_{H\times GL_1(F)}^{SO(U)} \pi^\prime$ $\otimes \chi$. By Lemma \ref{lemma induction},
this induced representation is irreducible.

Now $$\mathrm{Ind}_{H\times GL_1(F)}^{SO(U)} \pi^\prime \otimes \chi=\mathrm{Ind}_{G\times GL_1(F)}^{SO(U)} \mathrm{Ind}_{H\times GL_1(F)}^{G\times GL_1(F)}\pi^\prime \otimes \chi$$ 
is irreducible, and so 
$\mathrm{Ind}_{H\times GL_1(F)}^{G\times GL_1(F)}\pi^\prime \otimes \chi$ is also irreducible.

By Frobenius reciprocity we get that $$\mathrm{Hom}_{G\times GL_1(F)}(\pi\otimes \chi, \mathrm{Ind}_{H\times GL_1(F)}^{G\times GL_1(F)}\pi^\prime \otimes \chi)\simeq \mathrm{Hom}_{H\times GL_1(F)}(\pi^\prime \otimes \chi, \pi^\prime \otimes \chi),$$ because the restriction of $\pi \otimes \chi$ to $H\times GL_1(F)$ is $\pi^\prime \otimes \chi$. Since $\pi^\prime \otimes \chi$, $\pi\otimes \chi$, and $\mathrm{Ind}_{H\times GL_1(F)}^{G\times GL_1(F)}\pi^\prime \otimes \chi$ are all irreducible, we see that 
$\mathrm{Ind}_{H\times GL_1(F)}^{G\times GL_1(F)} \pi^\prime\otimes \chi\simeq \pi\otimes \chi$.

b) If $\pi^K\neq\{0\}$, then take a non-zero vector $v$ fixed by $K$. Then clearly $v\in \pi^{\prime K\cap H}$.

Assume now that $\pi^{\prime K\cap H}\neq \{0\}$. 
Let $K_0$ be an open compact subgroup of $F^\times$ such that $\chi(K_0)=1$.
Then using part a) we have
$$\{0\}\neq 
\mathrm{Ind}_{(K\cap H)\times K_0}^{K\times K_0}\pi^{\prime K\cap H} \otimes \chi^{K_0}=
\left (\mathrm{Ind}_{H\times GL_1(F)}^{G\times GL_1(F)}\pi^{\prime} \otimes \chi\right )^{K\times K_0}=\left (\pi \otimes \chi\right )^{K\times K_0}.$$ 
Since $\chi(K_0)=1$, the space $\left(\pi \otimes \chi\right )^{K\times K_0}$ is equal to $\pi^K$,
and so $\pi^K$ is also non-zero.
\end{proof}

%
%
%

\chapter{Upper Bound}

In this chapter we prove an upper bound on the number $N^K_{\mathrm{sd}}(\lambda)$ of self-dual cuspforms on $GL_N(\mathbb{A_Q})$.
The idea of the proof is to consider the functorial descent of each self-dual $\Pi$ to an appropriate classical group. 
First we collect some information about functoriality in Section \ref{upper tools}. The proof is then in Section \ref{upper proof}.

Throughout this chapter, all groups and representations are adelic, unless stated otherwise.

\section{Tools}\label{upper tools}

Throughout this chapter, let $K=K_\infty\times \prod K_p$ be a compact subgroup of $GL_N(\mathbb A)$ with $K_\infty=O(N)$.
Mostly we will be working with the principal congruence subgroup $K=K(M)$,
for which $K_p$ is the group of matrices $A\in GL_N(\mathcal O_p)$ satisfying $A\equiv I\pmod {p^{m(p)}}$
(when $m(p)=0$ this just means that $K_p=GL_N(\mathcal O_p))$, where $M=\prod p^{m(p)}$.

Let us first establish global analogues of the local results from Chapter 4.

\begin{proposition}\label{upper K fixed}
a) For every $K$ there is a constant $C_K$ (depending only on $K$) such that for every automorphic representation $\Pi$ of $GL_N(\mathbb A)$ one has $\dim\Pi^K\leq C_K$.

b) Let $\mathbf G$ be a quasisplit classical group.
For each $M\in\mathbb N$ there exists an arithmetic compact subgroup $K^\prime\subset G(\mathbb{A})$ such that:

For each self-dual cuspidal automorphic representation $\Pi$ of $GL_N(\mathbb {A})$ which descends to a cuspidal automorphic representation $\pi$ of $G(\mathbb {A})$ 
and satisfies $\Pi^{K(M)}\neq 0$, one has $\pi^{K^\prime}\neq 0$.
\end{proposition}

\begin{proof}
a) Take $\Pi$ such that $\Pi^K\neq 0$ and write it as a restricted tensor product $\Pi=\otimes\Pi_v$. 
Then $\Pi^K=\prod \Pi_v^{K_v}$, and so $\dim\Pi^K=\prod \dim\Pi_v^{K_v}$.
When $K_v$ is maximal and $\Pi_v$ is unramified (or when $v=\infty$), then $\dim\Pi_v^{K_v}=1$, and so this is just a finite product.

At each of the remaining finitely many places $v=p$, we have a constant $c_p=c_{p, K_p}$ (independent of $\Pi_p$) from Theorem \ref{Bernstein} such that $\dim\Pi_p^{K_p}\leq c_p$.
Let $C_K$ be the product of $c_p$ over these ramified places. Then $\dim\Pi^K\leq C_K$ and $C_K$ depends only on $K$.

b) Take $\Pi=\otimes\Pi_v$ such that $\Pi^K\neq 0$ and which descends to a cuspidal automorphic representation $\pi$ of $\mathbf G(\mathbb {A_Q})$.
When $K_p$ is maximal, $\Pi_p$ is unramified. Thus also $\pi_p$ is unramified and we can set $K^\prime_p$ to be a special maximal compact subgroup adopted to a standard minimal parabolic so that we have the corresponding Iwasawa decomposition (when $\mathbf G$ is split, we can take $K^\prime_p=\mathbf G(\mathcal O_p)$). Likewise, we can define $K^\prime_\infty$ to be a maximal compact subgroup of $\mathbf G(\mathbb R)$ (again chosen suitable to provide Iwasawa decomposition).

At the remaining places $v=p$, we use Corollary \ref{depth cor}, which gives us a 
compact open subgroup $K_p^\prime\subset G$ with $\pi_p^{K^\prime_p}\neq 0$. 
Then $K^\prime=K^\prime_\infty\times\prod K^\prime_p$ satisfies the conditions of the proposition.
\end{proof}

Let us next discuss to which classical groups $\mathbf G$ can $\Pi$ descend.
When $N=2n+1$ is odd, the only possibility is $\mathbf G=Sp_{2n}$. But when $N=2n$ is even, $\mathbf G$ can be split $SO_{2n}$ or $SO_{2n+1}$, 
or one of the infinitely many non-split quasisplit groups $SO_{2n}^\ast$. However, for a fixed compact subgroup $K$, only finitely many of these can occur:

\begin{lemma}\label{upper sk}
Let $K=K(M)$. There is a finite set $S_K$ of quasisplit classical groups $\mathbf G$ such that every self-dual cuspidal automorphic representation $\Pi$ of $GL_N(\aq)$ with $\Pi^K\neq\{0\}$ descends to some $\mathbf G\in S_K$. 
\end{lemma}

\begin{proof}
We just need to consider the case when $N=2n$. Let $\Pi$ be a self-dual cuspidal automorphic representation of $GL_N(\aq)$ with $\Pi^K\neq\{0\}$ and let $\chi$ be the central character of $\Pi$.
This is a character of the center $Z\simeq GL_1(\aq)$ of $GL_N(\aq)$.
Assume moreover that $\Pi$ descends to one of the non-split groups $SO_{2n}^\ast$. This group is then uniquely determined by the central character $\chi$.

We have $\chi(Z\cap K)=1$, which implies that the conductor of $\chi$ divides $M$.
By class field theory, there are only finitely many characters with conductor dividing $M$. Thus there are also only finitely many possibilities for the group $SO_{2n}^\ast$.
\end{proof}

\section{Statement and Proof of Theorem \ref{upper upper}}\label{upper proof}

We are now ready to prove the upper bound. 

For $\lambda>0$ we have defined
$N^K(\lambda)=N_{GL(N)}^K(\lambda)=\sum_{\lambda(\Pi)\leq\lambda} \dim \Pi^K$, where the sum is over
all cuspidal automorphic representations $\Pi$ of $GL_N(\mathbb A)$ such that the
Laplacian eigenvalue $\lambda(\Pi)$ is at most $\lambda$. 
Similarly define $N^K_{\mathrm{sd}}(\lambda)$ to be a similar sum $\sum_{\lambda(\Pi)\leq\lambda} \dim\Pi^K$, but this time ranging only over self-dual representations $\Pi$.

\begin{theorem}\label{upper upper}
Let $K=K(M)$.
There are positive constants $c$ and $\lambda_0$ (both depending on $K$) such that $N^K_{\mathrm{sd}}(\lambda)\leq c\lambda^{d/2}$ for all $\lambda>\lambda_0$.
Here $d=n^2+n$ for $N=2n+\varepsilon$ with $\varepsilon=0, 1$.
\end{theorem}

\begin{proof}
First of all, when $N=2$, we can just use the upper bound from (full) Weyl's Law for $GL_2$, since the dimension $d$ is 2 in both cases. Assume from now on that $N\neq 2$.

We need to estimate $$N^K_{\mathrm{sd}}(\lambda)=\sum_{\lambda(\Pi)\leq\lambda,\ \Pi\mathrm{\ sd}} \dim\Pi^K.$$ 

By Lemma \ref{upper sk}, there is a finite set $S=S_K$ such that each of our self-dual $\Pi$ descends to some $\mathbf G\in S$.
Hence we can write $N^K_{\mathrm{sd}}(\lambda)=\sum_{\mathbf G\in S}N^K_{\mathrm{sd}, \mathbf G}(\lambda)$, where $N^K_{\mathrm{sd}, \mathbf G}(\lambda)$ is the sum over all $\Pi$ which descend to $\mathbf G$. 
It thus suffices to obtain an upper bound for each $N^K_{\mathrm{sd}, \mathbf G}(\lambda)$.

By Proposition \ref{upper K fixed}a), there is $c_1$ such that $\dim\Pi^K\leq c_1$ for each such $\Pi$. Hence 
$$N^K_{\mathrm{sd}, \mathbf G}(\lambda)\leq c_1\cdot\#\{\Pi\ |\ \lambda(\Pi)\leq\lambda,\ \Pi^K\neq 0,\ \Pi\mathrm{\ descends\ to\ }\mathbf G\}.$$

Let $\Pi$ be such that $\lambda(\Pi)\leq\lambda,\ \Pi^K\neq 0$, and $\Pi$ descends to $\mathbf G$, and let $\pi$ be its descent, which is a cuspidal automorphic representation of $\mathbf G(\aq)$. 
By Theorem \ref{main laplace}, the Laplacian eigenvalue of $\Pi$ is $\lambda(\Pi)=c_G\lambda(\pi)+d_G$. Hence there is $c_2$ such that for sufficiently large $\lambda$ we have $\lambda(\pi)\leq c_2\lambda$.
By Proposition \ref{upper K fixed}b), there is a compact subgroup $K^\prime$ of $\mathbf G(\aq)$ such that $\pi^{K^\prime}\neq 0$. (Note that neither $c_2$ nor $K^\prime$ depend on $\Pi$.)

Thus we see that 
$$\#\{\Pi\ |\ \lambda(\Pi)\leq\lambda,\ \Pi^K\neq 0,\Pi\mathrm{\ descends\ to\ }\mathbf G\}\leq$$ 
$$\leq\#\{\pi\ |\ \pi\mathrm{\ repre\-sen\-tation\ of\ }\mathbf G(\aq),\ \lambda(\pi)\leq c_2\lambda,\ \pi^{K^\prime}\neq 0\}.$$
This is at most $N_\mathbf{G}^{K^\prime}(c_2\lambda)$. 

Since $N\neq 2$, $\mathbf G$ is not $SO_2^\ast$. Hence $\mathbf G^1(\mathbb R)$ is not compact and we can use Donnelly's upper bound from Proposition \ref{adelic donnelly}. It gives us $N_\mathbf{G}^{K^\prime}(c_2\lambda)\leq c_3\lambda^{d(G)/2}$ for $\lambda$ large enough and some $c_3$ as we needed. Here 
$d(G)$ is the dimension of the symmetric space for $\mathbf G$, given by the table preceding Proposition \ref{adelic donnelly}.
\end{proof}

%
%
%

\chapter{Lower Bound}

In this chapter we prove a lower bound on the number $N^K_{\mathrm{sd}}(\lambda)$ of self-dual cuspforms on $GL_N(\mathbb{A_Q})$.
The idea of the proof is similar to that of the upper bound (i.e., obtaining the self-dual representations as functorial lifts from classical groups), 
but we have to overcome two additional difficulties: first, a cuspidal automorphic representation 
on a classical group need not lift to a cuspidal representation on $GL_N$. 
We deal with this issue in Section \ref{lower noncusp}. And second, we need to obtain an estimate 
on the size of ($K$-fixed elements of) global packets. 
However, to obtain a lower bound it suffices to consider only the case when $K$ is a maximal compact subgroup at all places, which greatly simplifies dealing with the corresponding packets -- the result that we need in this case was already discussed in Lemma \ref{arthur lemma}.
The proof of the lower bound is then in Section \ref{lower proof}.

Let us note that in this chapter we need to use Weak Weyl's Law (formulated in Theorem \ref{weak weyl}) to obtain 
a lower bound on the number $N_G^K(\lambda)$ of cusp forms on a classical group $\mathbf G$. 
This seems not to have appeared in literature yet, but it seems to follow from other results towards Weyl's Law (see the remarks immediately preceding and following Theorem \ref{weak weyl}).

Throughout this chapter, all groups and representations are adelic, unless stated otherwise.

\section{Non-Cuspidal Lifts}\label{lower noncusp}

Let $\mathbf G=\mathbf G_{N^\prime}$ be a split group $SO_{N^\prime}$ or $Sp_{N^\prime}$ for ${N^\prime}=2n+1$ or ${N^\prime}=2n$, respectively (note that we are not considering the even special orthogonal groups). 
Denote by $N$ the size of the corresponding dual group, i.e., $N={N^\prime}-1=2n$ if ${N^\prime}=2n+1$ is odd and $N={N^\prime}+1=2n+1$ if ${N^\prime}=2n$ is even.

Take $\lambda>0$ and the compact subgroup $K=K_\infty\times \prod K_p$  of $\mathbf G(\mathbb A)$ with $K_p=\mathbf G(\mathcal O_p)$ and $K_\infty$ a fixed maximal compact subgroup of $\mathbf G(\mathbb R)$ (adopted to a standard Borel subgroup so that we have the corresponding Iwasawa decomposition). Denote by
$N_{G, \mathrm{\ nc}}^K(\lambda)$ the sum $\sum_{\lambda(\pi)\leq\lambda}\dim \pi^K$, ranging
only over cuspidal automorphic representations $\pi$ such that their lift $\Pi$ to $GL_{N}(\aq)$ is not cuspidal.

It follows from Theorem \ref{arthur 152} that 
such a representation $\Pi$ is an isobaric sum $\Pi=\Pi_1\boxplus\cdots\boxplus\Pi_k$, where $\sum N_i=N$ and each
$\Pi_i$ is an automorphic representation in the discrete spectrum of $GL_{N_i}(\aq)$. 

Moreover, all the $\Pi_i$ are self-dual and are lifts of automorphic representations $\pi_i$ in the discrete spectrum of $\mathbf G_{N^\prime_i}(\aq)$. 
Here again $N_i=N^\prime_i\pm 1$ depending on the parity of $N^\prime_i$ as at the beginning of this section.
Note that all the $N^\prime_i$ have the same parity as $N^\prime$.
By conjugating if necessary, we can assume that the isobaric sum is with respect to the standard Levi $M=\prod_i GL_{N_i}(\aq)$. 

We shall use these facts to estimate the number $N_{G, \mathrm{\ nc}}^K(\lambda)$.
For that we again need to relate Laplacian eigenvalues 
of $\pi$ and $\pi_i$.

\begin{lemma}\label{lower laplace}
Let $M=\prod_i GL_{N_i}(\aq)$ be a standard Levi subgroup of $GL_{N}(\aq)$.

a) There are constants $c_{i}>0$ and $d$ (independent of $\pi$) so that $\lambda(\pi)=\sum_i c_{i}\lambda(\pi_i)+d$ for every cuspidal automorphic representation $\pi$ of $\mathbf G_{N^\prime}(\aq)$ which lifts to an isobaric sum $\Pi=\Pi_1\boxplus\cdots\boxplus\Pi_k$ with respect to the Levi $M$. Here $\pi_i$ is a representation in the discrete spectrum of $\mathbf G_{N^\prime_i}(\aq)$ which lifts to $\Pi_i$.

b) There are constants $c>0$ and $e$ such that $\lambda(\pi_i)\leq c\lambda(\pi)+e$ for all such $\pi$ and for all $i$.
\end{lemma}

\begin{proof}
The proof of part a) is similar to the proof of Theorem \ref{main laplace}:
Let $\varphi$ and $\varphi_i$ be the Langlands parameters of $\pi_\infty$ and $\pi_{i, \infty}$. Let $\iota$ be the embedding $G^\vee\hookrightarrow GL$. Here we abuse notation a little by not distinguishing between the embeddings $\iota$ for groups of different size. 
Then $\iota\circ\varphi$ and $\iota\circ\varphi_i$ are Langlands parameters of $\Pi_\infty$ and $\Pi_{i, \infty}$. 
By the properties of an isobaric sum, we have $\iota\circ\varphi=\oplus_i\ \iota\circ\varphi_i$.

Let $\mu$ and $\mu_i$ be infinitesimal characters of $\pi_\infty$ and $\pi_{i, \infty}$. Then $\iota(\mu)$ and $\iota(\mu_i)$ are infinitesimal characters of $\Pi_\infty$ and $\Pi_{i, \infty}$. We see that $\iota(\mu)=\oplus_i\ \iota(\mu_i)$.
Evaluating this infinitesimal character on the Casimir element as in the proof of Theorem \ref{main laplace} proves our claim.

Part b) now easily follows from a) using the fact that each $\lambda(\pi_i)\geq 0$.
\end{proof}
%
%
%

We are now ready to prove the following proposition:

\begin{proposition}\label{lower nc}
Let $K=K_\infty\times \prod K_p$ be a compact subgroup of $\mathbf G_{N^\prime}(\mathbb A)$ with $K_p=\mathbf G_{N^\prime}(\mathcal O_p)$ and $K_\infty$ a fixed maximal compact subgroup of $\mathbf G_{N^\prime}(\mathbb R)$ (adopted to a standard Borel subgroup so that we have the corresponding Iwasawa decomposition).
There are positive constants $c$ and $\lambda_0$ such that $N_{G, \mathrm{\ nc}}^K(\lambda)\leq c\lambda^{d_0/2}$ for all $\lambda>\lambda_0$.

Here 
\begin{itemize}
\item $d_0=c=0$ if ${N^\prime}=1, 3$,

\item $d_0=n^2-n$ if ${N^\prime}=2n\geq 2$,

\item $d_0=n^2-n+2$ if ${N^\prime}=2n+1\geq 5$.
\end{itemize}
\end{proposition}

Note that in all of these cases, $d_0$ is strictly smaller than the dimension $d$ from Weak Weyl's Law for $\mathbf G$.

\begin{proof}
Denote by
$N_{G, \mathrm{\ nc,\ } M}^K(\lambda)$ the sum $\sum_{\lambda(\pi)\leq\lambda} \dim\pi^K$, ranging
only over cuspidal automorphic representations $\pi$ which lift to isobaric sums $\Pi=\Pi_1\boxplus\cdots\boxplus\Pi_k$ with respect to the Levi $M$. 
We see that $N_{G, \mathrm{\ nc}}^K(\lambda)$ is a finite sum $\sum_M N_{G, \mathrm{\ nc,\ } M}^K(\lambda)$, the sum ranging over standard Levis $M$. 
Thus it suffices to estimate $N_{G, \mathrm{\ nc,\ } M}^K(\lambda)$.

Since $K$ is maximal at every place, $\dim\pi^K\leq 1$ for every $\pi$. Thus 
$$N_{G, \mathrm{\ nc,\ } M}^K(\lambda)=\#\{\pi\ |\ \lambda(\pi)\leq\lambda,\ \pi^K\neq 0,$$ $$\pi\mathrm{\ lifts\ to\ }\Pi_1\boxplus\cdots\boxplus\Pi_k
\mathrm{\ with\ respect\ to\ }M\}.$$

For such a $\pi$, let $\Pi=\Pi_1\boxplus\cdots\boxplus\Pi_k$ be its functorial lift and let 
$\pi_i$ be the automorphic representations in the discrete spectrum of $G_{N_i^\prime}$ which lift to $\Pi_i$ (since each packet contains at most one $K$-fixed element by Lemma \ref{arthur lemma}, $\pi_i$ are uniquely determined by $\pi$). 
Note that $\pi_i$ is spherical at all places (with respect to suitable maximal compact subgroups).

By Proposition \ref{lower laplace} we have $\lambda(\pi_i)\leq c\lambda(\pi)+e$, where $c$ and $e$ are constants independent of $\pi$. Hence we see that $N_{G, \mathrm{\ nc,\ } M}^K(\lambda)$ is less than or equal to the product $\prod_i N_{G(N_i^\prime), \mathrm{\ disc}}^K(c\lambda+e)$,
where $N_{G(N_i^\prime), \mathrm{\ disc}}^K(c\lambda+e)$ is the corresponding sum over the discrete spectrum of $G(N_i^\prime)$.

The same upper bound holds for the discrete and cuspidal spectrum, and so we conclude that 
$$N_{G, \mathrm{\ nc,\ } M}^K(\lambda)\leq c_M\lambda^{\sum d_i/2},$$
where $c_M$ is a constant depending only on $M$ and $d_i$ is the dimension of the symmetric space for $G_{N_i^\prime}$.

When $N$ is even, $\sum d_i$ will be the largest when $k=2$ and $N_1=2$, $N_2=N-2$.
When $N$ is odd, $\sum d_i$ will be the largest when $k=3$ and $N_1=N_2=1$, $N_3=N-2$.
The proposition then follows from Table \ref{table dimensions}.
\end{proof}

We would expect to have an analogue of Proposition \ref{lower nc} for general $K$, not only for maximal one. The obstacle is our (at least the author's) present lack of control over the behaviour of $K$-spherical members of $A$-packets for general $K$.

\section{Statement and Proof of Theorem \ref{lower lower}}\label{lower proof}

Let us now prove the lower bound on the number of self-dual representations.

\begin{theorem}\label{lower lower}
Let $K$ be a principal congruence subgroup $K(M)$ of $GL_N(\aq)$.
There are positive constants $c$ and $\lambda_0$ such that $N^K_{\mathrm{sd}}(\lambda)\geq c\lambda^{d/2}$ for all $\lambda>\lambda_0$.
Here $d=n^2+n$ for $N=2n+\varepsilon$ with $\varepsilon=0, 1$.
\end{theorem}

\begin{proof}
Since the dimension $d$ depends only on $N$ and not on $K$, it suffices to prove the theorem for $K=K_\infty\times \prod K_p$ with $K_p=GL_N(\mathcal O_p)$ and $K_\infty=O(N)$.
We shall obtain the lower bound by considering lifts only from $\mathbf G=\mathbf G_{N^\prime}=SO_{2n+1}$ or $Sp_{2n}$, when $N=2n$ or $2n+1$, respectively.

Denote by $K^\prime=K^\prime_\infty\times \prod K^\prime_p$ a compact subgroup of $\mathbf G(\mathbb A)$ with $K^\prime_p=\mathbf G(\mathcal O_p)$ and $K^\prime_\infty$ a fixed maximal compact subgroup of $\mathbf G(\mathbb R)$ (adopted to a standard Borel subgroup so that we have the corresponding Iwasawa decomposition).
By Weak Weyl's Law (Theorem \ref{weak weyl}), we have (for $\lambda$ sufficiently large) $N^{K^\prime}_{G}(\lambda)\geq c_1\lambda^{d/2}$ for a positive constant $c_1$ and $d=n^2+n$ as above. 

Let $N^{K^\prime}_{G,\mathrm{\ cusp}}(\lambda)$ be the sum $\sum_{\lambda(\pi)\leq\lambda}\dim \pi^{K^\prime}$, ranging only over cuspidal automorphic representations $\pi$ such that their lift $\Pi$ to $GL_{N}(\aq)$ is cuspidal. 
Combining Proposition \ref{lower nc} with Weak Weyl's Law, we obtain $N^{K^\prime}_{G,\mathrm{\ cusp}}(\lambda)\geq c_2\lambda^{d/2}$ for a positive constant $c_2$.

Since the dimension of the space of $K^\prime$-fixed vectors of $\pi$ is 0 or 1, we see that 
$N^{K^\prime}_{G,\mathrm{\ cusp}}(\lambda)$ is equal to the number of such $K^\prime$-spherical representations $\pi$. 
Each of these representations lifts to a $K$-spherical cuspidal automorphic representation $\Pi$ with Laplacian eigenvalue $\lambda(\Pi)=c_G\lambda(\pi)+d_G$ by Theorem \ref{main laplace} (for certain constants $c_G, d_G$ independent of $\pi$).

As we observed in Lemma \ref{arthur lemma}, a global $A$-packet has at most one $K^\prime$-spherical element, and so the functoriality mapping $\pi\mapsto\Pi$ is an injection. We conclude that
$$N^K_{\mathrm{sd}}(c_G\lambda+d_G)\geq N^{K^\prime}_{G,\mathrm{\ cusp}}(\lambda)\geq c_2\lambda^{d/2},$$
which proves the theorem.
\end{proof}

Together with Theorem \ref{upper upper}, this finishes the proof of our main Theorem \ref{intro main}.

\

Note that when Weyl's Law is known for $G$ (which is the case when $N=2n$ and $\mathbf G=SO_{2n+1}$ by \cite{lv}) and $K$ is a maximal compact subgroup, all of the estimates in the proofs of Theorems \ref{upper upper} and \ref{lower lower} are in fact asymptotically equalities. 
Hence we obtain Weyl's Law (without an error term) for self-dual representations of $GL_{2n}(\aq)$ in this case as:

\begin{corollary}\label{lower cor}
Let $N=2n$ be even and $K=K_\infty\times \prod K_p$ with $K_p=GL_N(\mathcal O_p)$ and $K_\infty=O(N)$. Then $N^K_{\mathrm{sd}}(\lambda)= c\lambda^{d/2} + o(\lambda^{d/2})$ for
an explicit positive constant $c$ and $d=n^2+n$.
\end{corollary}

\begin{proof}
We just need to check all the estimates in the proofs of upper and lower bounds.

Let us first consider the case $N=2$. Let $\Pi$ be an automorphic representation of $GL_2(\aq)$ with central character $\omega$.
Its dual is then $\widetilde{\Pi}\simeq \Pi\otimes\omega^{-1}$.
If $\Pi$ is $K$-spherical (for maximal $K$ as above), the central character $\omega$ is trivial on $\mathcal O_p^\times$ for every $p$. Hence $\omega$ is an unramified character of $\aq^\times$, and so it is trivial (class field theory attaches to $\omega$ an extension of $\mathbb Q$ which is unramified everywhere, and so it is just the trivial extension).
We see that every $K$-spherical automorphic representation of $GL_2(\aq)$ is self-dual, and so our claim is just Weyl's Law in this case.

Take now $N\geq 4$. With notation as in the proof of Theorem \ref{upper upper}, we have $N^K_{\mathrm{sd}}(\lambda)=\sum_{\mathbf G\in S}N^K_{\mathrm{sd}, \mathbf G}(\lambda)$.
As we have seen in the proof of \ref{upper upper}, the main term in the sum on the right hand side is the one corresponding to $\mathbf G=SO_{2n+1}$, and so 
$N^K_{\mathrm{sd}}(\lambda)=N^K_{\mathrm{sd},\ SO(2n+1)}(\lambda)+ o(\lambda^{d/2})$.
From now on, let $\mathbf G=SO_{2n+1}$.
When $\Pi$ is unramified everywhere, $\dim\Pi^K=1$, and so 
$$N^K_{\mathrm{sd},\mathbf G}(\lambda)=\#\{\Pi\ |\ \lambda(\Pi)\leq\lambda,\ \Pi^K\neq 0,\ \Pi\mathrm{\ descends\ to\ }\mathbf G\}.$$

Let $K^\prime=K^\prime_\infty\times \prod K^\prime_p$ be a compact subgroup of $\mathbf G(\mathbb A)$ with $K^\prime_p=\mathbf G(\mathcal O_p)$ and $K^\prime_\infty$ a fixed maximal compact subgroup of $\mathbf G(\mathbb R)$ (adopted to a standard Borel subgroup so that we have the corresponding Iwasawa decomposition).

Take a $K$-spherical self-dual representation $\Pi$ and denote by $\pi$ its descent to 
$\mathbf G(\aq)$. Then $\pi$ is $K^\prime$-spherical, and by Lemma \ref{arthur lemma}, $\pi$ is unique. By Theorem \ref{main laplace}, the Laplacian eigenvalue of $\pi$ is $\lambda(\Pi)=c_G\lambda(\pi)+d_G$ (for suitable constants $c_G, d_G$).

Hence, with notation as in the proof of \ref{lower lower}, 
$N^K_{\mathrm{sd},\mathbf G}(c_G\lambda+d_G)=N^{K^\prime}_{G,\mathrm{\ cusp}}(\lambda)$ is the number of cusp forms of $\mathbf G(\aq)$ with Laplacian eigenvalue bounded from above by $\lambda$, whose lifts to $GL_N(\aq)$ are cuspidal.
By Proposition \ref{lower nc} and by Weyl's Law for $\mathbf G$, we have
$$N^{K^\prime}_{G,\mathrm{\ cusp}}(\lambda)=
N^{K^\prime}_{G}(\lambda)-N^{K^\prime}_{G,\mathrm{\ nc}}(\lambda)=c_1\lambda^{d/2} + o(\lambda^{d/2})$$ 
for an explicit constant $c_1$ (here $N^{K^\prime}_{G,\mathrm{\ nc}}(\lambda)$ denotes the number of cusp forms whose lift is non-cuspidal).

Putting everything together, we obtain 
$$N^K_{\mathrm{sd}}(c_G\lambda+d_G)=N^K_{\mathrm{sd},\mathbf G}(c_G\lambda+d_G)+ o(\lambda^{d/2})=N^{K^\prime}_{G,\mathrm{\ cusp}}(\lambda)+ o(\lambda^{d/2})=c_1\lambda^{d/2} + o(\lambda^{d/2}).$$
This finishes the proof when we set $c=c_1c_G^{-d/2}$.
\end{proof}

A similar statement for $K$-spherical self-dual representations of $GL_{2n+1}(\aq)$ would follow from Weyl's Law for $Sp_{2n}(\aq)$.
However, regardless of the parity of $N$, we are crucially relying on the assumption that $K$ is maximal everywhere. Otherwise we do not know how to obtain sufficiently precise estimates in several of the steps of the proof, namely, for the choice of a corresponding compact subgroup $K^\prime$, 
for the dimension of $K$- and $K^\prime$-fixed vectors, 
and for the sizes of global $A$-packets.



%
%
%

\bibliography{bibliogra}



%
%
%
%

\begin{vita}
V\'\i t\v ezslav Kala received Bachelor's (in 2007) and Master's (in 2009) degrees in Mathematics from Charles University in Prague, Czech Republic.
He has been a PhD student in the Department of Mathematics at Purdue University from August 2009 to August 2014. During his first three years at Purdue, 
he has been supported by the International Fulbright Science and Technology Award.
\end{vita}

\end{document}